\pdfoutput=1
\newif\ifpersonal
\newif\ifarxiv
\personaltrue 
\arxivtrue 
\RequirePackage[l2tabu, orthodox]{nag} 
\documentclass[12pt,a4paper,reqno,oneside]{amsart} 
\usepackage{calligra}
\usepackage{amsmath,amsthm,amssymb,mathrsfs,mathtools,bm,eucal,tensor} 
\usepackage{microtype,fixltx2e} 
\usepackage[scaled]{beramono,berasans}
\usepackage{enumerate,comment,braket,xspace,tikz-cd} 
\usepackage[all,cmtip]{xy} 
\usepackage[utf8]{inputenc} 
\usepackage[T1]{fontenc} 
\usepackage{lmodern}
\definecolor{linkcolor}{HTML}{005050}
\usepackage[centering,vscale=0.7,hscale=0.7]{geometry}
\usepackage{hyperref}
\usepackage[capitalize]{cleveref}
\usepackage{graphicx}
\usepackage{xparse}
\usepackage{url}
\usepackage[toc,page]{appendix}
\DeclareMathAlphabet{\mathcal}{OMS}{cmsy}{m}{n} 
\usepackage[backend=biber, style=numeric, sorting = nyt , giveninits=true,doi=false,url=false,maxbibnames=99]{biblatex} 
\usepackage{vmargin}
\setpapersize{A4}
\setmarginsrb{25mm}{10mm}{25mm}{10mm}%
{12mm}{10mm}{5mm}{10mm}

\usepackage{fancyhdr}
\theoremstyle{definition}
\newtheorem{theorem}{Theorem}[section]
\newtheorem{lemma}[theorem]{Lemma}
\newtheorem{corollary}[theorem]{Corollary}
\newtheorem{proposition}[theorem]{Proposition}
\newtheorem{definition}[theorem]{Definition}
\newtheorem{example}[theorem]{Example}

\newtheorem{thm-intro}{Theorem}

\newtheorem*{thm*}{Theorem}

\newtheorem*{lem*}{Lemma}

\newtheorem{cor-intro}[thm-intro]{Corollary}

\theoremstyle{definition}

\newtheorem{defin-intro}[thm-intro]{Definition}

\theoremstyle{remark}
\newtheorem{rem}{Remark}

\newtheorem{eg-intro}[thm-intro]{Example}
\newtheorem{rem-intro}[thm-intro]{Remark}
\numberwithin{equation}{section}

\ifpersonal
\newcommand*{\personal}[1]{\textcolor[rgb]{0.6,0.6,1}{(Personal: #1)}}
\newcommand*{\todo}[1]{\textcolor{red}{(Todo: #1)}}
\else
\newcommand*{\personal}[1]{\ignorespaces}
\newcommand*{\todo}[1]{\ignorespaces}
\fi

\newcommand{\Tot}{\textbf{Tot}}
\newcommand{\C}{\mathbb C}

\newcommand{\Z}{\mathbb Z}

\newcommand{\bbR}{\mathbb R}

\newcommand{\ft}{\mathfrak t}
\newcommand{\fg}{\mathfrak{g}}
\newcommand{\fq}{\mathfrak q}

\newcommand{\cD}{\mathcal D}

\newcommand{\cJ}{\mathcal J}
\newcommand{\cK}{\mathcal K}

\newcommand{\cM}{\mathcal M}

\newcommand{\cO}{\mathcal O}

\newcommand{\cT}{\mathcal T}

\newcommand{\cV}{\mathcal V}

\newcommand{\cX}{\mathcal X}
\newcommand{\cY}{\mathcal Y}

\DeclareFontFamily{U}{BOONDOX-calo}{\skewchar\font=45 }
\DeclareFontShape{U}{BOONDOX-calo}{m}{n}{<-> s*[1.05] BOONDOX-r-calo}{}
\DeclareFontShape{U}{BOONDOX-calo}{b}{n}{<-> s*[1.05] BOONDOX-b-calo}{}
\DeclareMathAlphabet{\mathcalboondox}{U}{BOONDOX-calo}{m}{n}

\newcommand{\bbH}{\mathbb H}

\newcommand{\sm}{\Sigma}
\newcommand{\cs}{\widetilde{\Sigma}}
\newcommand{\ts}{\widebar{\Sigma}}

\newcommand{\wv}{\widehat{V}}
\newcommand{\wpi}{\widehat{\pi}}
\newcommand{\Jac}{\mathrm{Jac}}
\newcommand{\Prym}{\mathrm{Prym}}
\newcommand{\Nm}{\mathrm{Nm}}
\newcommand{\Tr}{\mathrm{Tr}}
\newcommand{\Gr}{\text{Gr}}

\newcommand{\e}{\epsilon}

\makeatletter
\let\save@mathaccent\mathaccent
\newcommand*\if@single[3]{%
	\setbox0\hbox{${\mathaccent"0362{#1}}^H$}%
	\setbox2\hbox{${\mathaccent"0362{\kern0pt#1}}^H$}%
	\ifdim\ht0=\ht2 #3\else #2\fi
}
\newcommand*\rel@kern[1]{\kern#1\dimexpr\macc@kerna}
\newcommand*\widebar[1]{\@ifnextchar^{{\wide@bar{#1}{0}}}{\wide@bar{#1}{1}}}
\newcommand*\wide@bar[2]{\if@single{#1}{\wide@bar@{#1}{#2}{1}}{\wide@bar@{#1}{#2}{2}}}
\newcommand*\wide@bar@[3]{%
	\begingroup
	\def\mathaccent##1##2{%
		\let\mathaccent\save@mathaccent
		\if#32 \let\macc@nucleus\first@char \fi
		\setbox\z@\hbox{$\macc@style{\macc@nucleus}_{}$}%
		\setbox\tw@\hbox{$\macc@style{\macc@nucleus}{}_{}$}%
		\dimen@\wd\tw@
		\advance\dimen@-\wd\z@
		\divide\dimen@ 3
		\@tempdima\wd\tw@
		\advance\@tempdima-\scriptspace
		\divide\@tempdima 10
		\advance\dimen@-\@tempdima
		\ifdim\dimen@>\z@ \dimen@0pt\fi
		\rel@kern{0.6}\kern-\dimen@
		\if#31
		\overline{\rel@kern{-0.6}\kern\dimen@\macc@nucleus\rel@kern{0.4}\kern\dimen@}%
		\advance\dimen@0.4\dimexpr\macc@kerna
		\let\final@kern#2%
		\ifdim\dimen@<\z@ \let\final@kern1\fi
		\if\final@kern1 \kern-\dimen@\fi
		\else
		\overline{\rel@kern{-0.6}\kern\dimen@#1}%
		\fi
	}%
	\macc@depth\@ne
	\let\math@bgroup\@empty \let\math@egroup\macc@set@skewchar
	\mathsurround\z@ \frozen@everymath{\mathgroup\macc@group\relax}%
	\macc@set@skewchar\relax
	\let\mathaccentV\macc@nested@a
	\if#31
	\macc@nested@a\relax111{#1}%
	\else
	\def\gobble@till@marker##1\endmarker{}%
	\futurelet\first@char\gobble@till@marker#1\endmarker
	\ifcat\noexpand\first@char A\else
	\def\first@char{}%
	\fi
	\macc@nested@a\relax111{\first@char}%
	\fi
	\endgroup
}
\makeatother

\usetikzlibrary{decorations.markings} 
\tikzset{
  closed/.style = {decoration = {markings, mark = at position 0.5 with { \node[transform shape, xscale = .8, yscale=.4] {/}; } }, postaction = {decorate} },
  open/.style = {decoration = {markings, mark = at position 0.5 with { \node[transform shape, scale = .7] {$\circ$}; } }, postaction = {decorate} }
}

\DeclareMathOperator{\coker}{coker}

\DeclareMathOperator{\End}{End}

\DeclareMathOperator{\Hom}{Hom}

\title{Semi-polarized
meromorphic Hitchin and Calabi-Yau integrable systems}

\author{Jia Choon Lee and Sukjoo Lee}

\address{Peking University, Beijing International Center for Mathematical Research, Jingchunyuan Courtyard \#78, 5 Yiheyuan Road, Haidian District,
Beijing 100871, China}
\email{jiachoonlee@gmail.com}

\address{Department of Mathematics, University of Edinburgh, EH9 3FD, UK}
\email{Sukjoo.Lee@ed.ac.uk}

\begin{document}

\maketitle
\begin{abstract}
    It was shown by Diaconescu, Donagi and Pantev that Hitchin systems of type ADE are isomorphic to certain Calabi-Yau integrable systems.
    In this paper, we prove an analogous result in the setting of meromorphic Hitchin systems of type A 
    which are known to be Poisson integrable systems. We consider a symplectization of the meromorphic Hitchin integrable
system, which is a semi-polarized integrable system in the sense of
Kontsevich and Soibelman. On the Hitchin side, we show that the moduli space of unordered diagonally framed Higgs bundles forms an integrable system in this sense and recovers the meromorphic Hitchin system as the fiberwise compact quotient. Then we construct a family of quasi-projective Calabi-Yau threefolds and show that its relative intermediate Jacobian fibration, as a semi-polarized integrable system, is isomorphic to the moduli space of unordered diagonally framed Higgs bundles.

\end{abstract}
\tableofcontents

\section{Introduction}
\subsection{Introduction}
Since the seminal work of Hitchin \cite{hitchin1987self}\cite{hitchin1987stable}, Higgs bundles and their moduli spaces have been studied extensively. There have been numerous deep results on the moduli space of Higgs bundles on a smooth projective curve related to other areas of mathematics such as the $P=W$ conjecture \cite{de2012topology}\cite{de2019hitchin}, the fundamental lemma in the Langlands program \cite{ngo2006fibration}\cite{ngo2010lemme}, the geometric Langlands conjecture \cite{Kapustin_2007} and mirror symmetry \cite{Hausel_2003}\cite{Langlandsduality}. One of the striking properties of these moduli spaces is that they admit a holomorphic symplectic form and the structure of an integrable system, called the \textit{Hitchin system}. In particular, the generic fiber of an integrable system is an abelian variety which turns out to be the Jacobian or (generalized) Prym variety of an associated spectral or cameral curve. This picture generalizes to the meromorphic situation where we allow the Higgs field to have poles along some divisors. While the meromorphic Hitchin system is no longer symplectic, it is still Poisson and integrable with respect to the Poisson structure \cite{markman1994spectral}\cite{bottacin1995symplectic}. 

On the other hand, Donagi-Markman and Donagi-Diaconescu-Pantev (DDP) introduced in \cite{donagi1994cubics}\cite{donagi1996spectral}\cite{diaconescu2006geometric}\cite{diaconescu2006intermediate} integrable systems coming from some families of projective or quasi-projective Calabi-Yau threefolds, called \textit{Calabi-Yau integrable systems}. A generic fiber is a complex torus or an abelian variety \cite{diaconescu2006geometric}\cite{diaconescu2006intermediate}, now obtained as the intermediate Jacobian of a Calabi-Yau threefold in the family. 

It is shown in \cite{diaconescu2006intermediate} that for adjoint groups $G$ of type ADE, there is an isomorphism between $G$-Hitchin systems and suitable Calabi-Yau integrable systems, which we call the \textit{DDP correspondence}. An interesting aspect of the construction in \cite{diaconescu2006intermediate} is that although the relevant Calabi-Yau threefold is non-compact, the (a priori mixed) Hodge structure on its third cohomology happened to be pure of weight one up to Tate twist. Because of this, the corresponding intermediate Jacobian is a compact torus (in fact an abelian variety). Since the data of a weight 1 Hodge structure is equivalent to the data of an abelian variety, this isomorpshism can be rephrased as an isomorphism between variations of weight  1 Hodge structures equipped with the abstract Seiberg-Witten differential (see \cite[Section 2]{beck2017hitchin}).

It is worth mentioning that the origin of this story comes from physics, specifically, large $N$ duality \cite{diaconescu2006geometric}. Recently, the correspondence has also found its place in the study of T-branes in F-theory \cite{anderson2014t}\cite{anderson2017t}. 

The isomorphism between Hitchin and Calabi-Yau integrable systems has been generalized successfully to groups of type BCFG by the work of Beck et al.  \cite{beck2017hitchin}\cite{beck2019calabi}\cite{beck2020folding} using the technique of foldings. 

\subsection{Main results}
The goal of this paper is to extend the DDP correspondence to the setting of meromorphic $SL(n,\C)$-Hitchin system $h:\cM(n,D) \to B$ where $D$ is a reduced divisor of the base curve. The best case scenario will be to construct a family of non-compact Calabi-Yau threefolds over the same base $B$ and show that the associated Calabi-Yau integrable system is isomorphic to the meromorphic Hitchin system as Poisson integrable systems. However, since the deformation space of such non-compact Calabi-Yau's is strictly smaller than the base $B$, we do not expect to get a natural family which induces the Possion integrable system (see \cite[Chapter 7]{kontsevich2014wall}).

Instead, we consider the notion of \textit{semi-polarized integrable systems} introduced by Kontsevich-Soibelman \cite{kontsevich2014wall}. These are non-compact versions of symplectic integrable systems whose fiber is a semi-abelian variety, an extension of an abelian variety by an affine torus. The main advantage is that they canonically induce the Poisson integrable systems as their compact quotients. In Section 2, we study this structure from the Hodge theoretic viewpoint. Since the data of a semi-polarized semi-abelian variety is equivalent to the data of a semi-polarized $\Z$-mixed Hodge structure of type $\{(-1,-1), (-1,0), (0,-1)\}$ (see Appendix A), the semi-polarized integrable system can be described as a variation of $\Z$-mixed Hodge structures of such type with an abstract Seiberg-Witten differential as in the pure case \cite{diaconescu2006intermediate}.

The main objects on the Hitchin side are the moduli space of diagonally framed Higgs bundles (resp. unordered), introduced by Biswas-Logares-Pe\'{o}n-Nieto \cite{biswas2018symplectic}\cite{biswas2019symplecticGHiggs}\footnote{In \cite{biswas2019symplecticGHiggs}, what we call "diagonally framed" is referred to as "relatively framed" in [21].}, and we denote these moduli space by $\overline{\cM}^{\Delta}(n,D)$ (resp. $\cM^{\Delta}(n,D)$). The moduli space $\overline{\cM}^{\Delta}(n,D)$ is a subspace of the moduli space of framed Higgs bundles $\cM_F(n,D)$ whose object is a triple $(E, \theta, \delta)$ where $(E,\theta)$ is a $SL(n,\C)$-Higgs bundle and $\delta$ is a framing of $E$ at $D$. As the name suggests, an object in $\overline{\cM}^\Delta(n,D)$ is a framed Higgs bundle such that the residue of its Higgs field is diagonal with respect to the framing $\delta$. The unordered version $\cM^\Delta(n,D)$ is obtained as the quotient of $\overline{\cM}^\Delta(n,D)$ by $S_n^{|D|}$ where $S_n^{|D|}$ is the product of symmetric groups $S_n$ acting on the space of the framings by permuting the order of components. The following diagram summarizes the relation among the moduli spaces:
\begin{equation}
    \begin{tikzcd}
    \overline{\cM}^{\Delta}(n,D) \arrow[r,hook]\arrow[d,"q"]&\cM_F(n,D)\arrow[d,"f_1"]\\
    \cM^\Delta(n,D) \arrow[r,"f_2"]\arrow[rd,"h_\Delta"]&\cM(n,D)\arrow[d,"h"] \\
    &B
    \end{tikzcd}
\end{equation}
where $q:\overline{\cM}^\Delta(n,D)\to \cM^\Delta(n,D)$ is the quotient map, $f_1$ and $f_2$ are the maps of forgetting the framings and $h_{\Delta}:=h \circ f_2:\cM^{\Delta}(n,D) \to B$ is the Hitchin map on the moduli space of unordered diagonally framed Higgs bundles that we will study. In this paper, we will mainly work over the locus $B^{ur}\subset B$ of smooth cameral curves which are unramified over $D$ and have simple ramifications. In particular, for a triple $(E, \theta, \delta)$ over $b\in B^{ur}$, the residue of $\theta$ over $D$ has distinct eigenvalues. We shall write the restrictions as $\overline{\cM}^\Delta(n,D)^{ur}:=(h_{\Delta}\circ q)^{-1}(B^{ur}) $ and $\cM^\Delta(n,D)^{ur}:= h^{-1}_{\Delta}(B^{ur}) $. 

We will show that $\overline{\cM}^\Delta(n,D)^{ur}$ and $\cM^\Delta(n,D)^{ur}$ are symplectic using deformation theoretic arguments. They also carry a smooth semi-polarized integrable system structure over the locus $B^{ur}$. The following is the first result of the paper. 

\begin{theorem}(Proposition \ref{prop:symplectic}, Corollary \ref{hitchin fibration is S-PIS.})
    The moduli space of unordered diagonally framed Higgs bundle $\cM^{\Delta}(n,D)$ is symplectic. The Hitchin fibration
    \begin{equation*}
        h^{ur}_\Delta:\cM^{\Delta}(n,D)^{ur} \to B^{ur}
    \end{equation*} forms a smooth semi-polarized integrable system whose fiber is a semi-abelian variety.
\end{theorem}

In order to prove this, we study the fiber $(h^{ur}_{\Delta})^{-1}(b)$ over each $b\in B^{ur}$ via the spectral correspondence between unordered diagonally framed Higgs bundles on $\sm$ and framed line bundles on the associated spectral cover $ \overline{p}_b:\ts_b \to \sm$. The framed line bundles on $\ts_b$ are then parametrized by the Prym variety $\Prym(\ts^\circ_b, \sm^\circ)$ associated to the restricted spectral cover $\overline{p}^\circ_b: = \overline{p}_b|_{\ts_b^\circ}: \ts^\circ_b\to \sm^\circ$ where $\ts^\circ_b:= \ts_b\setminus \overline{p}_b^{-1}(D)$ and $\sm^\circ: = \sm\setminus D.$ More precisely, $\Prym(\ts^\circ_b,\sm^\circ)$ is a semi-abelian variety defined as the connected component of the identity element of the kernel of the punctured norm map $\Nm^\circ: \Jac(\ts^\circ_b) \to \Jac(\sm^\circ).$ 

\begin{proposition}(Proposition \ref{spectralcorrespondence}, Spectral correspondence) 
A generic fiber $h_\Delta^{-1}(b)$ is canonically isomorphic to the semi-abelian variety $\Prym(\overline{\sm}_b^\circ, \sm^\circ)$. In particular, the first homology $H_1(\Prym(\overline{\sm}_b^\circ, \sm^\circ))$ admits a $\Z$-mixed Hodge structure of type $\{(-1,-1),(-1,0),(0,-1)\}$.
\end{proposition}

On the Calabi-Yau side, we construct a family of Calabi-Yau threefolds $\pi:\cX \to B$ by using the elementary modification technique in \cite{smith2015quiver}. To produce the relevant Calabi-Yau integrable systems, we should restrict the family $\pi:\cX \to B$ to $B^{ur}$, denoted by $\pi^{ur}:\cX^{ur} \to B^{ur}$, whose fiber is smooth and its third homology admits a $\Z$-mixed Hodge structures of type $\{(-1,-1),(-1,0),(0,-1)\}$ up to Tate twist. Now, by taking fiberwise intermediate Jacobians, we obtain a family of semi-abelian varieties $\pi^{ur}:\cJ(\cX^{ur}/B^{ur}) \to B^{ur}$. The local period map induces an integrable system structure of this family.

The main result of the paper is to establish an isomorphism between the two semi-polarized integrable systems: 
\begin{theorem}(Theorem \ref{thm:main meromorphic DDP}) \label{thm:Main theorem}
There is an isomorphism of smooth semi-polarized integrable systems
\begin{equation}
\begin{tikzcd}
     \mathcal{J}(\mathcal{X}^{ur}/B^{ur}) \arrow[rr, "\cong"]  \arrow[rd, "\pi^{ur}"]  &&  
     \cM^\Delta(n,D)^{ur}
      \arrow[ld, "h_{\Delta}^{ur}"] \\
    & B^{ur}
\end{tikzcd}
\end{equation}
\end{theorem}

    The idea is to compare the variations of $\Z$-mixed Hodge structures associated to the two semi-polarized integrable systems, by using the gluing techniques in \cite{diaconescu2006intermediate}, \cite{beck2017hitchin}. To complete the proof, we check that the comparison map intertwines the abstract Seiberg-Witten differentials on each side.
    
    \subsection{Related work}
    
        The ideas of the spectral correspondence for unordered diagonally framed Higgs bundles and the infinitesimal study of their moduli spaces are drawn from \cite{biswas2018symplectic}. We follow their approach closely in Section \ref{sec: deformation theory}. However, we provide an improvement of their result in order to show that $\overline{\cM}^{\Delta}(n,D)^{ur}$ and $\cM^{\Delta}(n,D)^{ur}$ are symplectic which was not proved before. We also focus more on the Hodge structures of the relevant Hitchin fibers to prove Theorem \ref{thm:Main theorem}.
        
        A general construction of the moduli space of unordered diagonally framed Higgs bundles $\cM^{\Delta}(n,D)$ comes from symplectic implosion \cite{SymplecticImplosion} associated to the level group action on $\cM_F(n,D)$, viewed as the cotangent bundle of the moduli of framed bundles \cite{markman1994spectral}. One can obtain the Hitchin fibration over the full base $B$, but it is a stratified space and very singular which makes it difficult to control. Indeed, as we only need the smooth part for our main result, we focus on Higgs fields that are diagonalizable over $D$ throughout the paper. 
    
    Kontsevich-Soibelman proposed a different construction of the relevant Calabi-Yau integrable system as an affine conic bundle over a holomorphic symplectic surface containing a given spectral curve (see \cite{kontsevich2014wall}). This can be done by blowing up intersections of spectral curves and the preimage of the divisor $D$ in the total space of the twisted cotangent bundle $K_\sm(D)$. After removing the proper transform of the preimage, one gets the desired symplectic holomorphic surface. This model is birationally equivalent to the one we introduce in Section 4. 
    
    After the completion and submission of this work, we were informed by one of the referees that our construction of Calabi-Yau threefolds had already appeared in the work of Abrikosov  \cite[Section 6.2]{abrikosov2018potentials} (and also considered in \cite{smith2020floer}) of which we were not aware since they are constructed for different purposes. While both of our constructions are written in a slightly different way, they are the same natural generalization of Smith's elementary modification construction.

\subsection{Plan}
We first recollect the basics of integrable systems and introduce the notion of a semi-polarized integrable system in Section 2. In Section 3, we study the integrable system structure of the moduli space of unordered diagonally framed Higgs bundle. Also, we give both the spectral and cameral descriptions for completeness. In Section 4, we construct the semi-polarized Calabi-Yau integrable systems by using the technique of elementary modification. It is then followed by a Hodge theoretic computation. Finally, in Section 5, we give a proof of Theorem \ref{thm:Main theorem}.

\subsection{Notation}
\begin{itemize}
    \item $\Sigma$ - a projective smooth curve of genus $g\geq 1$. 
    \item $D$ - an effective divisor of $d$ reduced points for $d\geq 2$.
    \item $\Sigma^{\circ}$ - the complement of the divisor $D$ in $\Sigma$. 
    \item $\cM(n,D)$ - the moduli space of $K_{\sm}(D)$-twisted $SL(n,\C)$-Higgs bundles. All the Higgs bundles in the following moduli spaces are $K_\sm(D)$-twisted with structure group $SL(n,\C)$. 
    \item $\cM_F(n,D)$ - the moduli space of framed Higgs bundles.
    \item $\overline{\cM}^{\Delta}(n,D)$ - the moduli space of diagonally framed Higgs bundles.
    \item $\cM^{\Delta}(n,D)$ - the moduli space of unordered diagonally framed Higgs bundles. 
    \item $B=\oplus_{i=2}^n H^0(\Sigma, K_\Sigma(D)^{\otimes i})$ - the Hitchin base. 
    \item $B^{ur}\subset B$ - the subset consists of smooth cameral curves which are unramified over $D$ and have simple ramifications. Throughout the paper, we will always assume an element $b\in B$ is sitting in $B^{ur}.$
    \item $\overline{p}_b: \ts_b \to \Sigma$ - the spectral cover for $b\in B$.
    \item $\widetilde{p}_b: \widetilde{\sm}_b\to \sm$ - the cameral cover for $b\in B.$
    
\end{itemize}

\section{Semi-polarized integrable systems}

In this section, we recall the notion of a \textit{semi-polarized integrable system}, originally introduced in \cite{kontsevich2014wall}. This is a non-proper generalization of the notion of algebraic integrable system  \cite{hitchin1987stable} which provides a new way to view integrable systems in the Poisson setting. Similarly to an algebraic integrable system which can be associated with variations of polarized weight one Hodge structures, we also have a Hodge-theoretic description of semi-polarized integrable systems. To make the paper self-contained, we shall begin reviewing basics of algebraic integrable systems by following \cite{beck2017hitchin}\cite{beck2019calabi}.

\subsection{Integrable systems and variations of Hodge structures}\label{sec:2.1pure case}

\begin{definition}
    Let $(M^{2n},\omega)$ be a holomorphic symplectic manifold of dimension $2n$ and $B$ be a connected complex manifold of dimension $n$. A holomorphic map $\pi:M \to B$ is called an \textit{algebraic integrable system} if it satisfies the following conditions.
    \begin{enumerate}
        \item $\pi$ is proper and surjective;
        \item there exists a Zariski open dense subset $B^\circ \subset B$ such that the restriction 
        \begin{equation*}
            \pi^\circ:=\pi|_{M^\circ}:M^\circ \to B^\circ, \qquad M^\circ:=\pi^{-1}(B^\circ)
        \end{equation*} has smooth connected Lagrangian fibers and admits a relative polarization i.e. a relatively ample line bundle.
    \end{enumerate}
    In particular, if $B^\circ =B$, then $(M,\omega, \pi)$ is called a \textit{smooth} algebraic integrable system. 
\end{definition}

The second condition that a generic fiber is Lagrangian puts rather restrictive constraints on the geometry of the fiber. To see this, first consider $\ker(d\pi^\circ)$, the sheaf of vector fields on $M^\circ$ which are tangent to the fibers of $\pi^\circ$. Since the fibers of $\pi^\circ$ are Lagrangians, the holomorphic symplectic form $\omega$ induces an isomorphism $\ker(d\pi^\circ) \cong (\pi^\circ)^*T^\vee B^\circ$ via $v \mapsto \omega(v,-)$. By taking pushforward to $B^\circ$, we have an isomorphism of coherent sheaves $\pi^\circ_*\ker(d\pi^\circ) \cong \pi^\circ_*(\pi^\circ)^*T^\vee B^\circ$. In fact, one can apply the projection formula and see $ \pi^\circ_*(\pi^\circ)^*T^\vee B^\circ \cong T^\vee B^\circ$ because the fibers of $\pi^\circ$ are connected. Thus, the sheaf $\pi^\circ_*\ker(d\pi^\circ)$ is isomorphic to $T^\vee B^\circ$, hence locally free. We denote it by $\cV$ and call it a vertical bundle of $\pi^\circ$.

Next, choose a sufficiently small open subset $U \subset B^\circ$ and two local sections $u,v:U \to \mathcal{V}$ such that they are Hamiltonian vector fields $u=X_{{(\pi^\circ)}^*f}, v=X_{{(\pi^\circ)}^*g}$ for the functions $f, g:U \to \C$. As the fibers of $\pi^\circ$ are Lagrangians, we have $[u,v]=X_{\omega(u,v)}=0$. It implies that the Lie algebra $(\cV, [-,-])$ is abelian so that one can define a group action of $\cV$ on $M^\circ$ via the fiberwise exponential map. In other words, the flows of the vector fields along the fibers of $\pi^\circ$ corresponding to the sections of $\cV$ act on $M^\circ$ while preserving the fibers of $\pi^\circ$.

The submanifold
\begin{equation*}
    \Gamma=\{v \in \mathcal{V}|\, \exists \, x \in M^\circ \text{such that} \quad v \cdot x=x\}
\end{equation*}
forms a full lattice in each fiber and induces a family of abelian varieties $\mathcal{A}(\pi^\circ):=\mathcal{V}/\Gamma \to B^\circ$ which acts simply transitively on $\pi^\circ: M^\circ \to B^\circ$. Therefore, a generic fiber of $\pi:M \to B$ is non-canonically isomorphic to an abelian variety.

From now on, we will focus on smooth integrable systems ($B^\circ=B$). From the viewpoint of Hodge theory, a family of polarized abelian varieties can be obtained from a variation of weight 1 polarized $\Z$-Hodge structures $\mathsf{V}=(V_\Z, F^\bullet V_\cO, Q)$ over $B$ where $V_\cO:=V_\C \otimes \cO_{B}$ and $F^\bullet$ is the Hodge filtration. This is done by taking the relative Jacobian fibration so that we have the family
\begin{equation}
    p: \mathcal{J}(\mathsf{V}):=\Tot(V_\cO/(F^1V_\cO+V_\Z)) \to B
\end{equation}
whose vertical bundle is $\mathcal{V}:=V_\mathcal{O}/F^1V_\mathcal{O} \to B$.

A natural question is a condition for the family $p: \mathcal{J}(\mathsf{V}) \to B$ being an integrable system. In other words, we need a symplectic form on $\mathcal{J}(\mathsf{V})$ where fibers are connected Lagrangians. This can be achieved by the following theorem. 

\begin{theorem} \cite{beck2017hitchin}
\label{Thm: VHS and integrable system}
Let $\mathsf{V}=(V_\Z, F^{\bullet}V_\cO, Q)$ be a variation of weight 1 polarized $\Z$-Hodge structures over $B$ and $\nabla^{GM}$ be the Gauss-Manin connection on $V_\cO$.
Assume that there exists a global section $\lambda \in H^0(B, V_\mathcal{O})$ such that 
\begin{align*}
    \phi_\lambda:&TB \to F^1V_{\cO}\\
    & \mu \mapsto \nabla^{GM}_{\mu} \lambda
\end{align*}
is an isomorphism. Then the polarization $Q$ induces a canonical symplectic form $\omega_\lambda$ on $\mathcal{J}(\mathsf{V})$ such that the induced zero section becomes Lagrangian. Moreover, the symplectic form is independent of the polarization $Q$ up to symplectomorphisms.

\end{theorem}
Consider the dual variation of $\Z$-Hodge structure of $\mathsf{V}$,  $\mathsf{V}^\vee=\Hom_{\mathsf{VHS}}(\mathsf{V},\Z_B)(-1)$ where $\Z_B$ is the constant variation of $\Z$-Hodge structure of weight 0 and $(-1)$ is the Tate twist\footnote{We follow the convention used in \cite{peters2008mixed}. Given a Hodge structure $V$ of weight $m$, its $k$-th Tate twist $V(k)$ has weight $m-2k$. } by $-1$ over $B$. Recall that $\Hom_{\mathsf{VHS}}(\mathsf{V}, \Z_B)$ is a variation of $\Z$-Hodge structure of weight $-1$ and we shift the weight to be $1$ by taking the Tate twist. The polarization $Q$ identifies $\mathcal{V}=  V_\cO/F^1V_\cO$ with $F^1V_\cO^\vee$. Consider the compositions
\begin{equation}
    \iota:\mathcal{V} \xrightarrow{\psi_Q} F^1V_\cO^\vee \xrightarrow{\phi_\lambda^\vee} T^\vee B.
\end{equation}
where $\psi_Q$ is the identification induced by the polarization $Q$ and $\phi_\lambda^\vee$ is dual of $\phi_\lambda$. Then the lattice $V_\Z$ in $\mathcal{V}$ embeds into $T^\vee B$ as a Lagrangian submanifold. Therefore, we obtain a symplectic structure from the canonical one on $T^\vee B$ by descending to
$\cJ(V_\cO) \cong T^\vee B/\iota(V_\Z)$. We call such $\lambda$ an \textit{abstract Seiberg-Witten differential} \cite{beck2017hitchin}\cite{Donagi1997SWintegrablesystem}.

\subsection{Semi-polarized integrable systems and variations of mixed Hodge structures}\label{sec: semipolarized IS}

One can generalize the notion of an algebraic integrable system by allowing fibers to be non-proper. This is the main object of our study, first introduced in \cite{kontsevich2014wall}. We recall the definition in a form convenient for our story.

\begin{definition}
    Let $(M^{2n+2k},\omega)$ be a holomorphic symplectic manifold of dimension $2n+2k$ and $B$ be a connected complex manifold of dimension $n+k$. A holomorphic map $\pi:M \to B$ is called a \textit{semi-polarized integrable system} if it satisfies the following conditions.
    \begin{enumerate}
        \item $\pi$ is flat and surjective;
        \item there exists a Zariski open dense subset $B^\circ \subset B$ such that the restriction 
        \begin{equation*}
            \pi^\circ:=\pi|_{M^\circ}:M^\circ \to B^\circ, \qquad M^\circ:=\pi^{-1}(B^\circ)
        \end{equation*} has smooth connected Lagrangian fibers;
        \item each fiber of $\pi^\circ$ is a semi-abelian variety which is an extension of a $n$-dimensional polarized abelian variety by a $k$-dimensional affine torus.
        
    \end{enumerate}
    In particular, if $B^\circ =B$, then $(M,\omega, \pi)$ is called a \textit{smooth} semi-polarized integrable system. 
\end{definition}
Similar to algebraic integrable systems, the main example comes from a variation of torsion-free $\Z$-mixed Hodge structures (see \cite[Section 14.4]{peters2008mixed} for its definition). Let $\mathsf{V}=(V_\Z, W_\bullet V_\Z, F^\bullet V_\cO)$ be a variation of $\Z$-mixed Hodge structures of type $\{(-1,-1), (-1,0), (0,-1)\}$ over $B$ where $V_\cO:=V_\C \otimes \cO_B$ and $\Gr^W_{-1}V_\C$ is polarizable. In other words, we have 
\begin{itemize}
    \item $0=W_{-3} \subset W_{-2} \subset W_{-1}=V_\Z$
    \item $0= F^{1} \subset F^{0} \subset F^{-1}=V_\cO$
\end{itemize}
and can choose a relative polarization on $\Gr^W_{-1}V_\cO$. Throughout this paper, we choose a semi-polarization on $V_\Z$, a degenerate bilinear form $Q:V_\Z \times V_\Z \to \Z_B$ which yields the relative polarization on $\Gr^W_{-1}V_\cO$. We call it a variation of semi-polarized $\Z$-mixed Hodge structures. Moreover, one can obtain a semi-abelian variety from a $\Z$-mixed Hodge structure of type $\{(-1,-1), (-1,0), (0,-1)\}$ by taking the Jacobian (see Appendix).  Therefore, we have a family of semi-abelian varieties by taking the relative Jacobian fibration
\begin{equation}
    p: \mathcal{J}(\mathsf{V}):=\Tot(V_\cO/(F^0V_\cO+V_\Z)) \to B
\end{equation}
with its compact quotient  $p_{\text{cpt}}:\mathcal{J}_{\text{cpt}}(\mathsf{V}):=\Tot(W_{-1}V_\cO/(W_{-1}V_\cO \cap F^0V_\cO+V_\Z)) \to B$.

To define an abstract Seiberg-Witten differential, we consider the dual variation of $\Z$-mixed Hodge structures $\mathsf{V}^\vee=(V^\vee_\Z, W_\bullet V_\Z^\vee, F^\bullet V_\cO^\vee):=\Hom_{\text{VMHS}}(\mathsf{V}, \Z_B)$ of  $\mathsf{V}$. Note that we don't take a Tate twist so that it is of type $\{(0,1),(1,0),(1,1)\}$. Unlike the pure case (Section \ref{sec:2.1pure case}), the abstract Seiberg-Witten differential is defined as a global section of the dual vector bundle $V_\cO^\vee$.

\begin{definition}
    Let $\mathsf{V}=(V_\Z, W_\bullet V_\Z, F^\bullet V_\cO, Q)$ be a variation of semi-polarized $\Z$-mixed Hodge structures of type $\{(-1,-1), (-1,0), (0,-1)\}$ over $B$, and $\nabla^{GM}$ be the Gauss-Manin connection on $V_\cO$. We define an abstract Seiberg-Witten differential as a global section of the dual bundle $V^\vee_\cO$, $\lambda \in H^0(B, V^\vee_{\cO})$, such that the following morphism
    \begin{equation} \label{def:SW condition}
        \begin{aligned}
    \phi_\lambda : &TB \to F^{1}V_\cO^\vee\\
    & \mu \mapsto \nabla^{GM}_{\mu}\lambda
\end{aligned}
    \end{equation}
is an isomorphism.
\end{definition}

It is clear that the vertical bundle $\mathcal{V}$ of $\mathcal{J}(\mathsf{V}) \to B$ can be identified with $(F^1V_\cO^\vee)^\vee$ via the canonical non-degenerate pairing,  $V_\cO/F^0V_\cO \otimes F^1V_\cO^\vee \to \cO_B$. Consider the composition 
\begin{equation*}
    \iota: \mathcal{V} \to (F^{1}V^\vee_\mathcal{O})^\vee \xrightarrow{\phi_\lambda^\vee} T^\vee B
\end{equation*}
under which the lattice $V_\Z \subset \mathcal{V}$ embeds into  $T^\vee B$ as a Lagrangian submanifold. Similar to Theorem \ref{Thm: VHS and integrable system}, we obtain a symplectic form from the canonical one on $T^\vee B$ with Lagrangian condition on a generic fiber.
Moreover, the total space $\cJ(\mathsf{V})$ has a canonical Possion structure associated to the given symplectic form. As the action of the affine torus on $\cJ(\mathsf{V})$ is Hamiltonian, free and proper, the quotient space $\cJ_\textrm{cpt}(\mathsf{V})$ is a Poisson manifold \cite[Section 4.1.2]{kontsevich2014wall}. Thus, $\cJ_\textrm{cpt}(\mathsf{V})$ has a Poisson integrable system structure whose symplectic leaves are locally parametrized by $\phi_\lambda^{-1}(\Gr^W_2 V_\cO^\vee \cap F^1V_\cO^\vee$) (see \cite[Section 4.2]{kontsevich2014wall} for more details). This proves the following proposition.

\begin{proposition}\label{SWdifferentials and S-PIS}
    Let $\mathsf{V}=(V_\Z, W_\bullet V_\Z, F^\bullet V_\cO, Q)$ be a variation of semi-polarized $\Z$-mixed Hodge structures of type $\{(-1,-1), (-1,0), (0,-1)\}$ over $B$ and \\  $\lambda \in H^0(B, V^\vee_\cO)$ be the abstract Seiberg-Witten differential. Then, the relative Jacobian fibration
    \begin{equation}
    p: \mathcal{J}(\mathsf{V}):=\Tot(V_\cO/(F^0V_\cO+V_\Z)) \to B
\end{equation}
    forms a semi-polarized integrable system. In particular, the compact quotient $\cJ_\textrm{cpt}(\mathsf{V}) \to B$ admits a Poisson integrable system structure.  
\end{proposition}

\begin{rem}

The reason we take a global section of the dual vector bundle in the definition of abstract Seiberg-Witten differential is that, unlike the pure case, the semi-polarization $Q$ does not induce the canonical identification between $\mathsf{V}$ and $\mathsf{V}^\vee$. Moreover, this is also motivated by the geometric examples we will consider where $V_\Z$ and $V_\Z^\vee$ are torsion-free integral homology and cohomology of a non-singular quasi-projective variety, respectively. 
\end{rem}

\begin{rem}
In \cite{kontsevich2014wall}, Kontsevich and Soibelman introduce the notion of a central charge $Z \in H^0(B, V_\cO^\vee)$ which induces an local embedding of the base into $V_\cO^\vee$. It is equivalent to the data of an abstract Seiberg-Witten differential which suits our story better. 
\end{rem}

\section{Moduli space of diagonally framed Higgs bundles}

In this section, we will study the moduli space of (unordered) diagonally framed Higgs bundles and the associated Hitchin map as introduced in \cite{biswas2018symplectic}. In particular, we will give the spectral and Hodge theoretic description of the generic Hitchin fiber.
Then we prove that it is a semi-polarized integrable system in two different ways: using deformation theory and using abstract Seiberg-Witten differentials. As mentioned in Section 1, parts of this section will follow the approach of \cite{biswas2018symplectic}. For basic properties of Hitchin systems and spectral covers, we refer to \cite{donagi1996spectral}. 

\subsection{The moduli space of (unordered) diagonally framed Higgs bundles}
We fix $\sm $ to be a smooth projective curve of genus $g$, $D$ a reduced divisor on $\sm$ and $\sm^\circ : = \sm\setminus D.$ 
\begin{definition}
A framed $SL(n,\C)$-Higgs bundle on $\sm$ is a triple $(E,\theta, \delta) $, where $E$ is a vector bundle of rank $n$ with trivial determinant, $\delta: E_D\xrightarrow{\sim} \oplus _{i=1}^n\cO_D $ is an isomorphism, i.e. a framing at $D$, and $\theta\in \Gamma( \Sigma, \End_0(E)\otimes K_\Sigma(D) )$ is a traceless Higgs field. 
\end{definition}
A morphism between framed Higgs bundles $(E, \theta, \delta)$ and $(E', \theta', \delta')$ is a map $f:E\to E'$ such that $\delta \circ f|_D = \delta' $ and $\theta' \circ f = (f\otimes Id_{K_{\sm}(D)})\circ \theta$. 
\begin{rem}
A framed $GL(n,\C)$-Higgs bundle and $PGL(n,\C)$-Higgs bundle are defined in a similar way (see \cite[Section 2]{biswas2019symplecticGHiggs}). 
\end{rem}
In order to discuss moduli spaces, we first define the stability conditions we will be using. 
We shall follow the definition of stability conditions in \cite{biswas2018symplectic}. Essentially, the stability condition for a framed Higgs bundle is just the stability condition for a $K_{\sm}(D)$-twisted Higgs bundle. More precisely, we say that a framed Higgs bundle $(E, \theta, \delta)$ is stable (semistable respectively) if for every $\theta$-invariant proper subbundle $F\subset E$, that is, $\theta(F) \subset F\otimes K(D)$, we have $\mu (F)<\mu(E) $ ($\mu(F)\leq \mu(E)$ respectively). Here we write $\mu$  for the slope $\mu(E)  =\textrm{deg}(E)/\textrm{rk}(E).$

The following lemma and the next corollary can be found in \cite[Lemma 2.3]{biswas2018symplectic}. We record them here for future reference. Let $(E, \theta) $ and $(E, \theta') $ be $K_{\sm}(D)$-valued semistable Higgs bundles on $\sm$ with $\mu(E)= \mu(E').$ 

\begin{lemma}\label{rigidlemma}
Let $f: E\to E'$ be a $\cO_\sm$-modules homomorphism such that 
\begin{enumerate}
    \item $\theta'\circ f =(f\otimes Id_{K_{\sm}(D)})\circ \theta $,
    \item there is a point $x_0\in \sm$ such that $f|_{x_0}= 0$, 
\end{enumerate}
then $f$ vanishes identically. 
\end{lemma}

\begin{corollary}\label{no_automorphism}
    A semistable framed Higgs bundle admits no non-trivial automorphism. 
\end{corollary}
\begin{proof}
Indeed, suppose $(E, \theta, \delta)$ admits an automorphism $h$, then the morphism $h-Id_E$ vanishes on $D$. By the Lemma $\ref{rigidlemma}$ above, $h- Id_E $ vanishes identically or equivalently $h= Id_E.$ 
\end{proof}

We denote $\fg:= \mathfrak{sl}_n$ and $\fg_E:= \End_0(E) $. Let $\ft$ be the vector subspace of diagonal traceless $n\times n$ matrices and $\fq$ be the orthogonal complement of $\ft$ with respect to the Killing form, i.e. the vector subspace of $n\times n$ matrices whose diagonal entries are all zero. We have $\fg = \ft\oplus \fq$. Given a framing $\delta$ of $E$, we can define the $\delta$-restrictions to $D$ as the compositions:
\begin{align*}
    \fg_E \twoheadrightarrow \fg_E\otimes \cO_D \xrightarrow{Ad_{\delta} }\fg \otimes \cO_D \twoheadrightarrow \fq\otimes \cO_D \\
    \fg_E \twoheadrightarrow \fg_E\otimes \cO_D \xrightarrow{Ad_{\delta} }\fg \otimes \cO_D \twoheadrightarrow \ft\otimes \cO_D 
\end{align*}
where the maps $\fg\otimes \cO_D\twoheadrightarrow \fq\otimes \cO_D$ and $\fg\otimes \cO_D\twoheadrightarrow \ft\otimes \cO_D$ are given by the projections for the decomposition $\fg = \ft \oplus \fq.$ Here $Ad_{\delta}$ sends an element $s\in \fg_E\otimes \cO_D$ to $\delta\circ s\circ \delta^{-1}\in \fg \otimes \cO_D$. 

Given a framed bundle $(E, \delta)$, we define subsheaves $\fg'_E, \fg''_E \subset \fg_E$ as the kernels
\begin{align*}
    0 \to \fg'_E \to \fg_E \to \fq_D: = i_*\fq \to 0\\
    0\to  \fg''_E \to \fg_E \to \ft_D:= i_*\ft\to 0 
\end{align*}
where $i:D\hookrightarrow \sm$ is the inclusion. In other words, a section of endomorphism in $\fg'_E$ ($\fg''_E$ respectively) restricted to $p \in D$ is diagonal (anti-diagonal respectively) with respect to $\delta$.

\begin{definition}
We say that a framed Higgs bundle $(E,\theta, \delta)$ is diagonally framed if $\theta\in H^0(\sm, \fg'_E\otimes K_\sm(D))\subset H^0(\sm, \fg_E\otimes K_\sm(D))$. 
\end{definition}

By the results of \cite{simpson1994moduli}\cite{Simpson1994ModuliOR} \cite[Section 2]{biswas2018symplectic}, it is shown that the moduli space of semistable framed $SL(n,\C)$-Higgs bundles $\cM_F(n,D)$ exists as a fine moduli space that is a smooth irreducible quasi-projective variety. The moduli space we are interested in is the moduli space of semistable diagonally framed $SL(n,\C)$-Higgs bundle, denoted by $\overline{\cM}^{\Delta}(n,D)$. It is clear that $\overline{\cM}^\Delta(n,D)$ is a subvariety of $\cM_F(n,D).$  

\begin{rem}
Unless mentioned otherwise, we will assume all diagonally framed Higgs bundles are semistable with structure group $SL(n,\C)$ throughout the paper. 
\end{rem}

For each $p\in D$, there is a natural $S_n$-action on $\oplus _{i=1}^n\cO_p$ by permuting the order of the components
\begin{equation*}
    \sigma: \oplus _{i=1}^n\cO_p\xrightarrow{\sim}\oplus _{i=1}^n\cO_p, \quad (s_1,...,s_n) \mapsto (s_{\sigma(1)},...,s_{\sigma(n)}), \quad \textrm{where  }\sigma\in S_n.
\end{equation*}
For each $p\in D$, this induces a $S_n$-action on the space of framings
\begin{equation*}
    \sigma \cdot \delta  = \sigma \circ \delta : E|_p \to \oplus _{i=1}^n\cO_p \xrightarrow{\sigma }\oplus _{i=1}^n\cO_p . 
\end{equation*}
Hence, the moduli spaces $\overline{\cM}^{\Delta}(n,D)$ and $\cM_{F}(n,D)$ admit a $S_n^{|D|}$-action: for $\underline{\sigma}\in S_n^{|D|}$,
\begin{equation*}
    \underline{\sigma}: (E, \theta, \delta) \mapsto  (E, \theta, \underline{\sigma}\cdot \delta) , \quad \textrm{where} \quad \underline{\sigma}\cdot \delta: E|_D\to\oplus _{i=1}^n \cO_D\to \oplus _{i=1}^n\cO_D.
\end{equation*}
Since the group is finite, we can consider the quotient $\cM_F(n,D) / (S_n^{|D|})$. The effect of taking quotient is that, for a fixed Higgs bundle, framings that differ only in reordering of components will be identified. More precisely, a morphism between unordered framed Higgs bundles $(E,\theta,\delta)$ and $(E',\theta',\delta')$ is a map $f: E\to E'$ such that 
\begin{equation*}
    \delta\circ f|_{D} = \underline{\sigma} \circ \delta' \textrm{   for some   }\underline{\sigma}\in S_n^{|D|}, \quad \theta' \circ f = (f\otimes Id_{K_{\sm}(D)})\circ \theta.
\end{equation*}
In other words, $\cM_F(n,D) / (S_n^{|D|})$ now parametrizes unordered framed Higgs bundles. However, this group action is not free. In order to get a free action by $S_n^{|D|}$, we will assume that the associated spectral curve is smooth and unramified over $D$, or equivalently, the residue of $\theta $ at $D$ has distinct eigenvalues. More precisely, we define $B^{ur}$ to be the locus of smooth cameral curves (see Section \ref{sec: cameral description}) which are unramified over $D$ and have simple ramifications. Of course, the associated spectral curve for $b\in B^{ur}$ is automatically a smooth spectral curve that is unramified over $D$, and the necessity to work with smooth cameral curve with simple ramifications will be explained in Section 5. Moreover, we restrict to the subvariety $\overline{\cM}^\Delta(n,D)^{ur} := \overline{h}^{-1}_{\Delta}(B^{ur})$ where $\overline{h}_\Delta$ denotes the composition $\overline{\cM}^\Delta(n,D) \hookrightarrow \cM_F(n,D) \xrightarrow{f_1} \cM(n,D) \xrightarrow{h}B$ and $f_1$ denotes the forgetful map. 
\begin{lemma}
The $S_n^{|D|}$-action on $\overline{\cM}^\Delta(n,D)^{ur} $ is free. 
\end{lemma}
\begin{proof}
Consider $(E, \theta, \delta) \in \overline{\cM}^\Delta(n,D)^{ur} $ and suppose that there exists $\underline{\sigma}\in S_n^{|D|}$ and an isomorphism $\alpha: (E, \theta, \delta ) \to (E, \theta, \underline{\sigma}\circ \delta)$. The compability condition $\delta\circ \alpha|_D = \underline{\sigma}\circ \delta $ implies that $\delta\circ \alpha|_D \circ \delta^{-1}  = \underline{\sigma}$, while the compatibility condition $\theta \circ \alpha =( \alpha \otimes Id_{K_{\sm}(D)})\circ \theta$ restricted to $D$ is equivalent to $\theta_\delta\circ \underline{\sigma} = \underline{\sigma} \circ \theta_\delta $ where $\theta_\delta := \delta^{-1}\theta|_D \delta.$ The last relation  $\theta_\delta\circ \underline{\sigma} = \underline{\sigma} \circ \theta_\delta$ is clearly not possible as $\theta_\delta$ is diagonal with distinct eigenvalues at each $p\in D.$ 
\end{proof} 

Since the $S_n^{|D|}$-action on $\overline{\cM}^\Delta(n,D)^{ur} $ is finite and free, we get a geometric quotient $\cM^\Delta(n,D)^{ur} : = \overline{\cM}^{\Delta}(n,D)^{ur}/(S_n^{|D|}) $. The variety $\cM^{\Delta}(n,D)^{ur}$ parametrizes unordered diagonally framed Higgs bundles.

Clearly, there is a morphism $f_2:\cM^\Delta(n,D)^{ur} \to \cM(n,D)^{ur}: = h^{-1}(B^{ur})$ by forgetting the framings. For our purpose of proving Theorem \ref{thm:Main theorem}, we will need to study the composition of the forgetful map $f_2$ and the Hitchin map $h$, denoted by $h_{\Delta}^{ur} : \cM^{\Delta}(n,D)^{ur}\xrightarrow{f_2} \cM(n,D)^{ur} \xrightarrow{h^{ur}} B^{ur}$. We summarize the relation among the moduli spaces over $B^{ur}$:
\begin{equation}
    \begin{tikzcd}
    \overline{\cM}^{\Delta}(n,D)^{ur} \arrow[r,hook]\arrow[d,"q"]&\cM_F(n,D)^{ur}\arrow[d,"f_1"]\\
    \cM^\Delta(n,D)^{ur} \arrow[r,"f_2"]\arrow[rd,"h_\Delta^{ur}"]&\cM(n,D)^{ur}\arrow[d,"h^{ur}"] \\
    &B^{ur}
    \end{tikzcd}
\end{equation}
where $\cM_F(n,D)^{ur}: =(h\circ f_1)^{-1}(B^{ur}).$

\subsection{Spectral correspondence}
We explain the spectral correspondence for unordered diagonally framed Higgs bundles (see Proposition \ref{spectralcorrespondence}). After that, we describe the Hodge structures of a generic Hitchin fiber which will be used in the proof of the main theorem. 
\begin{definition}
Let $D$ be an effective reduced divisor on $C.$ A $D$-framed line bundle on a curve $C$ is a pair $(L, \beta)$ where $L$ is a line bundle and $\beta : L|_D\xrightarrow{\sim } \cO_D $ is an isomorphism. 
\end{definition}
\begin{rem}
Unless mentioned otherwise, we will call $(L, \beta)$ a framed line bundle whenever the divisor $D$ is clear from the context.
\end{rem}

\begin{proposition}\label{framedlinebundles}  
Let $C$ be a projective smooth curve and $D$ a reduced divisor on $C$. Let $C^\circ=C\setminus D$, $j:C^\circ\to C$ and $i: D\to C$ be the natural inclusions. The isomorphism classes of degree 0 framed line bundles on $C$ are parametrized by the generalized Jacobian
\begin{equation}
    \Jac(C^\circ): = \frac{H^0(C ,\Omega_{C} (\log D))^\vee}{H_1(C^\circ, \Z)}.
\end{equation}
\end{proposition}
\begin{proof}
We follow some of the ideas in the work of Arapura-Oh \cite{arapura1997abel}. By duality, we can identify 
\begin{equation*}
    \Jac(C^\circ) = \frac{H^0(C, \Omega_C (\log D))^\vee }{H_1(C^\circ, \Z)} \cong \frac{H^1(C,\cO(-D))  }{H^1(C, D , \Z)}
\end{equation*}
Consider the exponential sequence
\begin{equation}\label{compactly-supported-exp-sequence}
    0\to  j_!\underline{\Z} \to \cO_C(-D) \xrightarrow{\exp({2\pi i}(-) )} \cO_C^*(-D) \to 0
\end{equation}
where $\cO_C^*(-D)$ is defined as the subsheaf of $\cO_C^*$ consisting of functions with value 1 on $D.$ It induces a long exact sequence 
\begin{equation*}
\begin{aligned}
    \cdots \to & H^1( C,  j_!\Z) \cong H^1(C,D,\Z) \to H^1(C , \cO_C(-D) ) \to H^1(C, \cO_C^*(-D)) \\ \xrightarrow{\text{c}_1} & H^2(C, j_!\Z) \cong H^2(C,D,\Z) \to H^2(C, \cO_C(-D) ) \to H^2(C, \cO_C^*(-D)) \to \cdots
\end{aligned}
\end{equation*}
where the map $\text{c}_1: H^1(C,\cO_C^*(-D))\to H^2(C,j_!\Z) \cong H^2(C,D,\Z)\cong H^2(C,\Z)\cong \Z$ is the first Chern class map. The group $H^1(C, \cO_C^*(-D))$ naturally parametrizes all framed line bundles. 
Indeed, the sheaf $\cO_C^*(-D)$ sits in a short exact sequence 
\begin{equation*}
    1\to \cO_C^*(-D) \to \cO_C^*\to i_*\C^*\to 1
\end{equation*}
which induces a quasi-isomorphism $\cO_C^*(-D) \to F^\bullet:= [\cO_C^* \to i_*\C^*]$ and hence an isomorphism $H^1(C, \cO_C^*(-D))\cong \bbH^1(C, F^\bullet).$ By choosing a \v{C}ech covering $(U_\alpha)$, a 1-cocycle in $Z^1(U_\alpha, F^\bullet)$ is a pair of $f_{\alpha\beta}\in H^0(U_{\alpha\beta}, \cO_C^*) $ and $\eta_{\alpha}\in H^0(U_\alpha, i_*\C^*)$ such that $ \eta_\alpha /\eta_\beta = f_{\alpha\beta} |_D .$ The data $f_{\alpha \beta}$ represents a line bundle. By assumption, $f_{\alpha\beta}|_D= 1$ implies that $\eta_\alpha|_D =\eta_\beta|_D\in \C^*$. Since a framing of a line bundle at a point is equivalent to a choice of a non-zero complex number, $(\eta_\alpha)$ defines a framing of the line bundle at $D.$ In other words, the pair $(f_{\alpha\beta}, \eta_\alpha)$ represents a framed line bundle, and a class in $\bbH^1(C, F^\bullet)$ represents an isomorphism class of the framed line bundle. 

In particular, we find that 
\begin{equation*}
   \Jac(C^\circ)\cong \frac{H^1(C,\cO(-D))}{H^1(C,D,\Z)}\cong  \ker (\text{c}_1: H^1(C, \cO_C^*(-D) ) \to \Z ) 
\end{equation*}
which paramatrizes degree 0 framed line bundles. 
\end{proof}
We will apply the previous discussion to $C= \overline{\sm}_b$ , a spectral curve of $\sm$ corresponding to $b\in B^{ur}$. 
\begin{rem}
Unless mentioned otherwise, we will omit the the subscript $b$ in $\overline{\sm}_b$ and $\overline{\sm}_b^\circ$ in this section for convenience, as it is irrelevant to our discussion.
\end{rem}
Since we are mainly interested in $SL(n,\C)$-Higgs bundles, we will need to consider the Prym variety of the spectral cover $\overline{p}: \overline{\sm}\to \sm$. The norm map $\Nm:\Jac(\ts)\to \Jac(\Sigma)$ induces a morphism of short exact sequences 
\begin{equation*}
\begin{tikzcd}
    0 \arrow[r] &(\C^*)^{nd-1} \arrow[r]\arrow[d]& \Jac (\ts^\circ)\arrow[d,"\Nm^\circ"] \arrow[r]& \Jac(\ts ) \arrow[r]\arrow[d,"\Nm"]& 0\\
    0 \arrow[r] &(\C^*)^{d-1} \arrow[r]& \Jac (\sm^\circ) \arrow[r]& \Jac(\sm ) \arrow[r]& 0
    \end{tikzcd}
\end{equation*}
where $d=|D|$ and $\Nm^\circ:\Jac(\ts^\circ) \to \Jac(\sm^\circ)$ is defined by taking norms on line bundles and determinants on framings. Recall that $\Nm(L) = \det(\overline{p}_*L)\otimes \det(\overline{p}_*\cO_{\overline{\sm}})^\vee$ and for a framed line bundle $(L,\beta)\in \Jac(\overline{\sm}^\circ)$, the natural framing 
\begin{equation*}
    \overline{p}_*L|_x\xrightarrow{\sim} \bigoplus_{y\in \overline{p}^{-1}(x)} L_y \overset{\sim}{\underset{ \beta }{\longrightarrow}} \bigoplus_{y\in \overline{p}^{-1}(x)} \cO_y 
\end{equation*}
induces a framing on $\det(\overline{p}_*L)|_x$ over each $x \in D$. Also, there is a natural framing on $\det(\overline{p}_*\cO_{\overline{\sm}})^\vee|_x$ induced from the identity $Id: \cO_{\overline{\sm}}|_{\overline{p}^{-1}(x)}\to \cO_{\overline{\sm}}|_{\overline{p}^{-1}(x)}$. Both framings determine a framing on $\Nm(L)$ and hence the map $\Nm^\circ$. 

By taking the kernel of this morphism, we get a commutative diagram:
\begin{equation}\label{abeliangp_diagram}
\begin{tikzcd}
    0 \arrow[r] &(\C^*)^{(n-1)d} \arrow[r]\arrow[d]& \Prym (\ts^\circ,\sm^\circ) \arrow[r]\arrow[d]& \Prym(\ts ,\sm) \arrow[r]\arrow[d]& 0\\
    0 \arrow[r] &(\C^*)^{nd-1} \arrow[r]\arrow[d]& \Jac (\ts^\circ)\arrow[d,"\Nm^\circ"] \arrow[r]& \Jac(\ts ) \arrow[r]\arrow[d,"\Nm"]& 0\\
    0 \arrow[r] &(\C^*)^{d-1} \arrow[r]& \Jac (\sm^\circ) \arrow[r]& \Jac(\sm ) \arrow[r]& 0
    \end{tikzcd}
\end{equation}
where $\Prym(\ts^\circ, \sm^\circ):= \ker(\Nm^\circ)^0$, the connected component of the identity element of the kernel of $\Nm^\circ$.

\begin{proposition}(Spectral  correspondence \cite{biswas2018symplectic}).\label{spectralcorrespondence}
\\ 
For a fixed $b\in B^{ur}$, there is a one-to-one correspondence between degree zero framed line bundles on $\ts_b$ and unordered diagonally framed Higgs bundles on $\sm$. Moreover, the following results hold: 
\begin{enumerate}
    \item The fiber $h_{\Delta, GL(n)}^{-1}(b)$ is isomorphic to $\Jac(\ts_b^\circ)$; 
    \item The fiber $h^{-1}_{\Delta,SL(n)}(b)$ is isomorphic to $\Prym(\ts_b^\circ,\sm^\circ)$.
\end{enumerate}

\end{proposition}
\begin{proof}
For simplicity, we assume $D=\{x\}, \overline{D}= \overline{p}^{-1}(x)$ in this proof. Let $L$ be a line bundle on $\ts_b$ and $(E, \theta)$ a Higgs bundle on $\sm$. Recall that \cite{beauville1989} there is a bijection between line bundles on $\ts_b$ and Higgs bundles on $\sm$ 
\begin{equation}
    \begin{tikzcd}
    \textrm{Line bundle $L$ on $\ts_b$  } \arrow[r, "\overline{p}_*",bend left]& \textrm{Higgs bundle $(E,\theta)$ on $\sm$}\arrow[l, "coker(\overline{p}^*\theta - \lambda Id)", bend left]
    \end{tikzcd}
\end{equation}
where $\lambda $ denotes the tautological section of $K_\sm(D)$. It remains to verify the bijection on framings. 

Pushing forward a $\overline{D} $-framed line bundle $(L,\beta)$ gives an unordered framed Higgs bundle $(\overline{p}_*L,\overline{p}_*\lambda,\delta) $ where
\begin{equation*}
    \delta: E|_x \xrightarrow{\sim }\bigoplus_{y\in p^{-1}(x) } L_y \overset{\sim}{\underset{ \beta }{\longrightarrow}} \bigoplus_{y\in p^{-1}(x)} \cO_y
\end{equation*}
is well-defined as an unordered framing. With respect to the unordered framing, the Higgs field $\overline{p}_*\lambda$ is diagonal as $\theta|_x:= \overline{p}_*\lambda$ defines multiplication by $\lambda_i$ on each eigenline $L_i$. 

Conversely, given an unordered diagonally framed Higgs bundle $(E, \theta, \delta)$, since we assume that $\theta|_x$ has distinct eigenvalues, for each $\lambda_i\in \overline{p}^{-1}(D)$, the natural composition
\begin{equation*}
    \ker(\overline{p}^*\theta- \lambda_i Id) \to E|_x\to \coker(\overline{p}^*\theta- \lambda_i Id)
\end{equation*}
is an isomorphism. The assumption that $\theta|_x$ is diagonal with respect to $\delta$ implies that there is a component $\cO_x\xhookrightarrow{\alpha_i} \oplus_{i=1}^n\cO_x$ such that 
\begin{equation*}
    \begin{tikzcd}
    \ker(\overline{p}^*\theta -\lambda_i Id)\arrow[r] & E|_x\\
    \cO_x \arrow[r,hook,"\alpha_i"]\arrow[u,"\cong"]&\oplus_{i=1}^n\cO_x\arrow[u,"\cong"]
    \end{tikzcd}
\end{equation*}
In particular, we get a framing $\cO_x \xrightarrow{\sim } \coker(\overline{p}^*\theta- \lambda_iId)$ for each $\lambda_i.$ 

Finally, claims (1), (2) follow from Proposition \ref{framedlinebundles}.
\end{proof}

\subsubsection{Hodge structures}
Recall that since $\ts^\circ$ is non-compact, $H^1(\ts^\circ, \Z)$ carries the $\Z$-mixed Hodge structure whose Hodge filtration is given by
\begin{equation}
    F^0 = H^1(\ts^\circ, \C) \supset F^1= H^0(\ts, {\Omega^1_{\ts}}(\log D)) \supset F^2=0.
\end{equation}
This induces the mixed Hodge structure on $(H^1(\ts^\circ, \Z))^\vee$ which is isomorphic to $H_1(\ts^\circ, \Z)$.

The Hodge filtration of this dual mixed Hodge structure is given by
\begin{equation*}
      F^{-1}= H^1(\ts^\circ, \C)^\vee\supset  F^{0} = \left(\frac{H^1(\ts^\circ, \C)}{H^0(\ts, {\Omega^1_{\ts}}(\log D))}\right)^\vee \supset F^1 = 0 
\end{equation*}
Note that the weight filtration on $H_1(\ts^\circ, \Z)$ is 
\begin{equation*}
    W_{-3}= 0\subset W_{-2}= \Z ^{nd-1} \subset W_{-1} =H_1(\ts^\circ, \Z).
\end{equation*}
Thus we can define as in \cite{carlson1979extensions} the Jacobian of this Hodge structure as 
\begin{equation}
    J(H_1(\ts^\circ,\Z) ):=  \frac{H_1(\ts^\circ,\C)}{F^{0}+H_1(\ts^\circ, \Z) }
\end{equation}

\begin{lemma}
There is an isomorphism between
\begin{equation*}
    J(H_1(\ts^\circ,\Z) )\cong   \Jac(\ts^\circ)
\end{equation*}
\end{lemma}
\begin{proof}
By definition, we have
\begin{equation*}
    J(H_1(\ts^\circ,\Z) )=  \frac{H_1(\ts^\circ,\C)}{F^{0}+H_1(\ts^\circ, \Z) }  = \frac{F^{-1}}{F^0+ H_1(\ts^\circ, \Z)} \cong \frac{H^0(\ts, {\Omega^1_{\ts}}(\log D))}{H_1(\ts^\circ, \Z)}.
\end{equation*}
\end{proof}

\noindent Taking the first integral homology of every term in the diagram (\ref{abeliangp_diagram}), we get 
\begin{equation*}
\begin{tikzcd}
    0 \arrow[r] &(\Z)^{(n-1)d} \arrow[r]\arrow[d]& H_{\Delta,SL(n)} \arrow[r]\arrow[d]& H_{SL(n)} \arrow[r]\arrow[d]& 0\\
    0 \arrow[r] &(\Z)^{nd-1} \arrow[r]\arrow[d]& H_1 (\ts^\circ,\Z)\arrow[d,"\Nm^\circ"] \arrow[r]& H_1(\ts,\Z ) \arrow[r]\arrow[d,"\Nm"]& 0\\
    0 \arrow[r] &(\Z)^{d-1} \arrow[r]& H_1 (\sm^\circ,\Z) \arrow[r]& H_1(\sm ,\Z) \arrow[r]& 0
    \end{tikzcd}
\end{equation*}
where we define 
\begin{align}
    H_{\Delta,SL(n)}&: = H_1(\Prym (\ts^\circ,\sm^\circ),\Z) \cong \ker (\Nm^\circ: H_1(\ts^\circ,\Z)\to H_1(\sm^\circ,\Z)),\\
    H_{SL(n)}&:= H_1(\Prym(\ts ,\sm),\Z) \cong \ker (\Nm: H_1(\ts,\Z)\to H_1(\sm,\Z)).
\end{align}

Since the norm map is a morphism of mixed Hodge structures and taking the Jacobian is functorial, we immediately get the following result.

\begin{corollary}\label{Hodge structure of Prym}
    The Prym lattice $H_{\Delta, SL(n)}$ is torsion free and admits the $\Z$-mixed Hodge structure of type $\{(-1,-1), (-1,0), (0,-1)\}$ induced by the map $H_1(\Nm^\circ):H_1(\ts^\circ,\Z) \to H_1(\sm^\circ, \Z)$. In particular, the Jacobian $J(H_{\Delta, SL(n)})$ is isomorphic to $\Prym(\ts^\circ,\sm^\circ)$.
\end{corollary}

\begin{rem}
Note that the $\Z$-mixed Hodge structure of the above type on $H_{\Delta, SL(n)}$ is equivalent to the data of semi-abelian variety $J(H_{\Delta, SL(n)})$. A review of this correspondence is included in Appendix \ref{appendix}. 
\end{rem}

The Prym lattice $H_{\Delta, SL(n)}$ admits a sheaf-theoretic formulation which will be needed in later sections. Consider the trace map $\overline{p}_*\underline {\Z} \xrightarrow{\Tr} \underline{\Z} $ which is given by 
\begin{equation*}
    \begin{aligned}
        \Tr(U):&\overline{p}_*\underline{\Z}(U) \to \underline{\Z}(U),\quad 
        & (s_1,...,s_n)\mapsto \sum_{i=1}^n s_i
    \end{aligned}
\end{equation*}
if $U$ is away from the ramification divisor. Now, consider the short exact sequence of sheaves
 \begin{equation}\label{prymses}
            0 \to \cK\to \overline{p}_*\underline{\Z} \xrightarrow{\Tr} \underline{\Z} \to 0
\end{equation}
where $\cK$ is defined to be the kernel of the trace map.

        The morphism $H^1(\sm, \overline{p}_*\underline{\Z})\cong H^1(\ts,\Z) \to H^1(\sm, \Z)$ induces the norm map on the Jacobians. To see this, note that we have a morphism of exact sequences
        \begin{equation*}
        \begin{tikzcd}
            0\arrow[r]& \overline{p}_*\underline{\Z}\arrow[r]\arrow[d,"\Tr"]& \overline{p}_*\cO_{\ts} \arrow[r, "\overline{p}_*(\exp)"]\arrow[d,"\Tr"]&\overline{p}_*\cO_{\ts}^*\arrow[d,"\det"]\arrow[r]&0 \\
            0\arrow[r] & \underline{\Z}\arrow[r]& \cO_{\sm}\arrow[r, "\exp"]&\cO_{\sm}^*\arrow[r]&0
        \end{tikzcd}
        \end{equation*}
        where the right vertical morphism is the determinant morphism $\det: \overline{p}_*\cO^*_{\ts}\to \cO^*_\sm$. This induces the long exact sequence
        \begin{equation*}
        \begin{tikzcd}
            ...\arrow[r]& H^1(\sm,\overline{p}_*\underline{\Z})\arrow[r]\arrow[d]& H^1(\sm,\overline{p}_*\cO_{\ts}) \arrow[r]\arrow[d]&H^1(\sm,\overline{p}_*\cO_{\ts}^*)\arrow[d]\arrow[r,"\mathrm{c}_1"]&... \\
            ...\arrow[r] & H^1(\sm,\Z)\arrow[r]& H^1(\sm,\cO_{\sm})\arrow[r]&H^1(\sm,\cO^*_{\sm})\arrow[r,"\mathrm{c}_1"]&...
        \end{tikzcd}
        \end{equation*}
        and the right vertical morphism restricted to $\ker(\mathrm{c}_1)$ is exactly the norm map on the Jacobians. Since $\sm$ and $\ts$ are compact, we will also write $\Nm: H^1_c(\ts, \Z)=H^1(\ts,\Z)\to H^1(\sm,\Z) = H^1_c(\sm, \Z)$. 
        
        The short exact sequence (\ref{prymses}) induces the long exact sequence:
        \begin{align*}
        0&\to H^0_c(\sm,\cK) \to H^0_c(\sm, \overline{p}_*\Z_{\ts})\cong H^0_c(\ts, \Z)\to  H^0_c(\Sigma,\Z)\nonumber \\
            &\to H^1_c(\sm, \cK)\to H^1_c(\Sigma, \overline{p}_*\Z)\cong H^1_c(\ts, \Z)   \to H^1_c(\Sigma,\Z)
        \end{align*}
        Since the cokernel of the map $H^0_c(\ts,\Z)\to H^0_c(\Sigma,\Z)$ is torsion and $H^1_c(\overline{\sm},\Z)$ is torsion-free, it follows that the maximal torsion free quotient $H^1_c(\sm , \cK)_{\textrm{tf}}:= H^1_c(\sm , \cK)/H^1_c(\sm , \cK)_{\textrm{tors}} $ can be identified as follows
        \begin{equation*}
            H^1_c(\sm, \cK)_{\textrm{tf}} \cong \ker (H^1_c(\ts,\Z)\xrightarrow{\Nm} H^1_c(\Sigma,\Z)) \cong \ker (H_1(\ts,\Z)\xrightarrow{\Nm} H_1(\Sigma,\Z))
        \end{equation*}
        by Poincar\'{e} duality.

        Note that we could have used cohomology instead of compactly supported cohomology since the curve $\sm$ is compact, but the above argument also works for the noncompact curve $\sm^\circ$. In particular, the same argument proves the following lemma.
        \begin{lemma}\label{rem:Prym sheaf}
        There is an isomorphism of torsion free abelian groups
                \begin{align*} \label{prym sheaf}
            H^1_c(\sm^\circ, \cK|_{\sm^\circ})_{\textrm{tf}} &\cong \ker (H^1_c(\ts^\circ,\Z)\xrightarrow{\Nm^\circ} H^1_c(\Sigma^\circ,\Z))\\
            &\cong \ker (H_1(\ts^\circ,\Z)\xrightarrow{\Nm^\circ} H_1(\Sigma^\circ,\Z))\\
            &\cong H_{\Delta,SL(n)}.
        \end{align*}
        where $\Nm^\circ: H^1_c(\ts^\circ, \Z)\cong H^1_c(\sm^\circ, \overline{p}^\circ_*\Z) \to H^1_c(\sm^\circ,\Z)$ is induced by the trace map and $\overline{p}^\circ:= \overline{p}|_{\ts^\circ}: \ts^\circ \to \sm^\circ$. In particular, up to Tate twist, this becomes an isomorphism of $\Z$-mixed Hodge structures
        \begin{equation*}
            H^1_c(\sm^\circ, \cK|_{\sm^\circ})_{\textrm{tf}}(1) \cong H_{\Delta,SL(n)}
        \end{equation*}
        \end{lemma}
        Here we again denote by $\Nm^\circ$ the morphism induced by the trace map because it induces the norm map on their Jacobians, which can be seen by applying the same argument as above to the exponential sequence of the form (\ref{compactly-supported-exp-sequence}). Note also that $H^1_c(\sm^\circ, \cK|_{\sm^\circ})_{\textrm{tf}}$ is isomorphic to $H^1(\sm,D, \cK)_{\textrm{tf}} $.

\subsection{Deformation theory} \label{sec: deformation theory}
In this section, we show that the moduli space of diagonally framed Higgs bundle $\overline{\cM}
^\Delta(n,D)$ is symplectic. For the following discussion in this section, we fix a diagonally framed Higgs bundle $(E,\theta, \delta).$ Recall that we assume $b\in B^{ur}$ which means that the associated cameral curve is smooth, unramified over $D$, and has simple ramification. In particular, the residue of $\theta$ at $D$ is diagonal with distinct eigenvalues with respect to the framing $\delta.$ Denote by $\sm[\e]$ the fiber product $\sm\times \textrm{Spec}(\C[\e])$. 
\begin{definition}
An infinitesimal deformation of diagonally framed Higgs bundle is a triple $(E_\epsilon, \theta_\epsilon, \delta_\epsilon)$ such that
\begin{itemize}
    \item $E_\epsilon$ is a locally free sheaf on $\sm[\e]$,
    \item $\theta_\e\in H^0(\sm[\e],\fg'_{E_\e,D[\e]}\otimes p_{\sm}^*K_\sm(D)) $,
    \item $\delta_\e: E|_{D[\e]}\to \oplus_{i=1}^n \cO_{D[\e]}$ is an isomorphism,
    \item $(E_\e, \theta_\e, \delta_\e)|_{D\times 0}\cong (E, \theta, \delta) $,
\end{itemize}
where as before $\fg'_{E_\e,D[\e]}$ is defined as the kernel of the map $\fg_{E_\e}\twoheadrightarrow \fq \otimes \cO_{D[\e]}$ induced by $\delta_\e $ and $p_\sm : \sm[\e]\to \sm $ denotes the natural projection.  

\end{definition}

\begin{proposition}\label{tangentspace}
The space of infinitesimal deformations of a diagonally framed Higgs bundle $(E,\theta, \delta)$ is canonically isomorphic to $\bbH^1 (C^\bullet) $ where 
\begin{equation}
     C^\bullet:  C^0= \fg_E(-D)\xrightarrow{ [\cdot, \theta]} C^1 = \fg'_E\otimes K_\sm(D) 
\end{equation}
\end{proposition}
\begin{proof}
Recall that \cite{markman1994spectral} the space of infinitesimal deformation of a framed Higgs bundles $(E, \theta, \delta)$ is canonically isomorphic to $\bbH^1(C^\bullet_F)$ where 
\begin{equation}
    C^\bullet_F:  C_F^0= \fg_E(-D) \xrightarrow{ [\cdot, \theta]} C^1_F=\fg_E\otimes K_\sm(D) .
\end{equation}
Choose a \v{C}ech cover $U:= \{U_\alpha\}$ of $\sm$ which induces cover $U[\e]:=( U_\alpha[\epsilon])$ of $\sm[\epsilon]$. Imposing further the condition that the Higgs bundles are diagonally framed 
implies that $\theta\in H^0(\sm,\fg_E' \otimes K_{\sm}(D)) \subset H^0(\sm, \fg_E \otimes K_{\sm}(D))$. Suppose that a 1-cocycle ($\dot{f}_{\alpha\beta}, \dot{\varphi}_\alpha)$ in $Z^1(U[\e], C^\bullet_F) $  represents an infinitesimal deformation of $(E, \theta, \delta)$ as framed Higgs bundles where $\dot{f}_{\alpha\beta}\in H^0(U_{\alpha\beta}[\epsilon],\fg_E(-D)) $ and $\dot{\varphi_{\alpha}}\in H^0(U_\alpha[\e],\fg_E\otimes p^*_{\sm}K_{\sm}(D))$. Then ($\dot{f}_{\alpha\beta}, \dot{\varphi}_\alpha)$ is an infinitesimal deformation of $(E, \theta, \delta)$ as diagonally framed Higgs bundles if and only if $\dot{\varphi}_\alpha\in H^0(U_\alpha[\e],\fg_E'\otimes p^*_{\sm}K_{\sm}(D))$. Hence, it follows that $\bbH^1(C^\bullet)$ parametrizes the infinitesimal deformations of diagonally framed Higgs bundles.

\end{proof}

Recall that the Serre duality says that $\bbH^1( C^\bullet) \xrightarrow{\sim} (\bbH^1(\check{C}^\bullet) )^\vee $ where 
\begin{equation}
    \check{C}^\bullet : (\fg_E')^\vee \otimes \cO_{\sm}(-D) \xrightarrow{[-,\theta]^t} \fg_E^\vee\otimes K_{\sm}(D)
\end{equation}
is the Serre dual to $C^\bullet$. Combining the Serre duality isomorphism with the isomorphism in the next proposition, we get a non-degenerate skew-symmetric pairing on $\bbH^1(C^\bullet).$

\begin{proposition}\label{deformationcomplex}
There is a canonical isomorphism
\begin{equation}
    \bbH ^1( \check{C}^\bullet) \cong \bbH^1 (C^\bullet). 
\end{equation}
\end{proposition}
\begin{proof}
We consider an auxiliary complex\footnote{The complex $C_1^\bullet$ here coincides with the complex $"\mathscr{C}_\bullet^\Delta"$ that is defined in \cite[Section 5]{biswas2018symplectic}.  } \begin{center}
    $C_1^\bullet: \fg''_E \xrightarrow{[\cdot, \theta]} \fg_E\otimes K_\sm(D)$
\end{center}
and show that this is quasi-isomorphic to $C^\bullet$ and isomorphic to $\check{C}^\bullet$. First, consider the morphism of complexes $t: C^\bullet \to C_1^\bullet$:
\begin{equation*}
    \begin{tikzcd}
    C^\bullet\arrow[d,"t"] &  \fg_E\otimes \cO_{\sm}(-D)\arrow[r]\arrow[d, "t_0"] & \fg'_E \otimes K_{\sm}(D)\arrow[d, "t_1"]    \\
    C_1^\bullet&  \fg''_E  \arrow[r]& \fg_E \otimes K_{\sm}(D)
    \end{tikzcd}
\end{equation*}
Both $t_0$ and $t_1$ are injective. The diagram clearly commutes away from $D$, hence commutes everywhere. In particular, around $D$, choose an open subset $U$ that trivializes all the bundles, we see that the maps become the natural maps
\begin{equation*}
    \begin{tikzcd}
       \ft(-D)\oplus \fq(-D) \arrow[r]\arrow[d, "t_0|_U"] &  (\ft \oplus \fq(-D ))\otimes K_\sm(D)|_U\arrow[d, "t_1|_U"] \\
       \ft(-D)\oplus \fq \arrow[r]& (\ft \oplus \fq )\otimes K_\sm(D)|_U 
    \end{tikzcd}
\end{equation*}
where we abuse notations by denoting $\ft$ and $\fq$ the trivial bundles with fibers $\ft$ and $\fq$, respectively.
The cokernel of $t$ is 
\begin{equation*}
    \coker(t) :  \fq_D \xrightarrow{[\cdot, \theta]|_D } \fq_D\otimes K_\sm(D) 
\end{equation*}
\begin{lemma}
\begin{equation*}
    \bbH^i(\coker(t) ) = 0 , \textrm{  for all }i. 
\end{equation*}
\end{lemma}
\begin{proof}
Since the complex is supported at $D$, it reduces to a complex of $\C$-vector spaces. Assume $D$ consists of a single point for simplicity. The complex reduces to 
\begin{equation*}
     \fq \xrightarrow{[\cdot, \theta]|_D} \fq.
\end{equation*}
Recall our assumption that the associated spectral curve is unramified over $D$. The restriction $\theta|_D$ of the Higgs field to $D$ is a diagonal matrix with distinct eigenvalues with respect to $\delta$. In particular, $\theta|_D$ is regular and semisimple, so its centralizer  \linebreak $Z_\fg(\theta|_D) = \{x \in g | [x, \theta|_D]=0\}$ is a Cartan subalgebra and coincides with $\ft$. Since $\ker([\cdot, \theta]|_D:\fg\to \fg) = Z_{\fg}(\theta|_D) =\ft $ which intersects $\fq$ trivially, it follows that the restricted map $([\cdot, \theta]|_D)|_{\fq}: \fq\to \fq$ is an isomorphism. Hence, all the cohomologies of the complex $\coker(t)$ must be zero.

\end{proof}
The long exact sequence induced by $0\to C^\bullet \to C^\bullet_1\to \coker(t) \to 0$ is:
\begin{align*}
    0&\to \bbH^0 (C^\bullet) \to \bbH^0(C^\bullet_1) \to \bbH^0(\coker(t)) =0\\
    &\to \bbH^1 (C^\bullet) \to \bbH^1(C^\bullet_1) \to \bbH^1(\coker(t))=0 \to ...
\end{align*}
and hence $\bbH^0(C^\bullet)\cong \bbH^0(C^\bullet_1)$ and $\bbH^1(C^\bullet) \cong \bbH^1(C^\bullet_1)$.  

Finally, we claim that there is an isomorphism of complexes $C^\bullet_1 \cong \check{C}^\bullet$
\begin{equation} \label{isom_of_complex}
    \begin{tikzcd}
    C_1^\bullet \arrow[d,"\cong"]&   \fg''_E\arrow[r]\arrow[d, "r_0"] & \fg_E \otimes K_{\sm}(D)\arrow[d, "r_1"]    \\
    \check{C}^\bullet&  (\fg'_E)^\vee\otimes \cO_\sm(-D) \arrow[r]& (\fg_E)^\vee \otimes K_{\sm}(D)
    \end{tikzcd}
\end{equation}
The map $r_0$ is defined as follows. Consider the  composition of morphisms 
\begin{equation}\label{composition of morphisms}
    \fg''_E \hookrightarrow \fg_E \xrightarrow{\sim} \fg_E^\vee \hookrightarrow (\fg'_E)^\vee \to (\fg'_E)^\vee\otimes \cO_D.
\end{equation}
where the isomorphism $\fg_E \to \fg_E^\vee$ is given by the trace pairing. If we know that this composition is zero, then we will get a map
\begin{equation*}
    r_0: \fg''_E \to \ker((\fg'_E)^\vee \to (\fg'_E)^\vee\otimes \cO_D) = (\fg'_E)^\vee \otimes \cO_\sm(-D).
\end{equation*}
Away from $D$, the map (\ref{composition of morphisms}) is clearly zero. Around $D$, we can find an open subset $U$ such that each sheaf in the composition is trivial, then 
\begin{equation*}
    \begin{tikzcd}
    \fg''_E|_U \arrow[r]\arrow[d,"\cong"]&\fg_E|_U\arrow[r]\arrow[d,"\cong"]
    &\fg_E^\vee|_U \arrow[r]\arrow[d,"\cong"]&(\fg'_E)^\vee|_U \arrow[r]\arrow[d,"\cong"] &((\fg'_E)^\vee\otimes \cO_D)|_U\arrow[d,"\cong"]\\
    \ft(-D)\oplus \fq \arrow[r]& \ft\oplus \fq \arrow[r]& \ft^\vee\oplus \fq^\vee \arrow[r] & \ft^\vee\oplus \fq^\vee(D) \arrow[r]&(\ft^\vee\otimes \cO_D) \oplus (\fq ^\vee \otimes \cO_D(D))
     \end{tikzcd}
\end{equation*}
 Each component of the bottom row clearly composes to zero, hence the whole composition is zero. Locally over $U$, the map $r_0:\fg_E'' \to \left( \fg_E' \right)^\vee \otimes \cO_\Sigma(-D)$ is induced by the trace pairing: $\ft\xrightarrow{\sim} \ft^\vee$ and $\fq\xrightarrow{\sim} \fq^\vee$,
 \begin{equation*}
     r_0|_U: \fg_E''|_U\cong \ft(-D)\oplus \fq \xrightarrow{\sim} \ft^\vee (-D) \oplus \fq^\vee \cong (\fg_E')^\vee\otimes \cO_\sm(-D) |_U
 \end{equation*}
which is an isomorphism. Away from $D$, it is clear that $r_0$ is also an isomorphism. It follows that $r_0$ is an isomorphism everywhere. 

The commutativity can be argued in the same way. Again, the diagram commutes away from $D$. Around $D$, the bundles trivialize and we get the diagram
\begin{equation*}
    \begin{tikzcd}
    \ft(-D) \oplus \fq \arrow[r]\arrow[d]& (\ft\oplus \fq) \otimes K_\sm(D)|_U\arrow[d]\\
    \ft^\vee (-D) \oplus \fq^\vee \arrow[r]&(\ft^\vee\oplus \fq^\vee) \otimes K_\sm(D)|_U
    \end{tikzcd}
\end{equation*}
which commutes on the nose. 

All of this together gives 
\begin{equation}
    \bbH^1(C^\bullet) \cong \bbH^1(C^\bullet_1) \cong \bbH^1(\check{C}^\bullet) .
\end{equation}
as claimed.
\end{proof}

Let $\omega_\Delta:\bbH^1(C^\bullet) \times \bbH^1(C^\bullet) \to \C$ be the non-degenerate skew-symmetric pairing induced by Serre duality and the isomorphism in Proposition \ref{deformationcomplex}.

\begin{proposition}
The nondegenerate 2-form $\omega_\Delta$ is closed. \label{symplecticform}
\end{proposition}
\begin{proof}
Consider the following inclusion of complexes $C^\bullet \xhookrightarrow{u} C^\bullet_F$:
\begin{equation*}
    \begin{tikzcd}
    C^\bullet\arrow[d,"u"] &  \fg_E\otimes \cO_{\sm}(-D)\arrow[r]\arrow[d, "u_0"] & \fg_E' \otimes K_{\sm}(D)\arrow[d, "u_1"] \\
    C_{F}^\bullet&  \fg_E \otimes \cO_{\sm}(-D) \arrow[r]& \fg_E \otimes K_{\sm}(D)
    \end{tikzcd}
\end{equation*}
where as before $C^\bullet_F$ is the complex whose first hypercohomology controls the deformations of the framed Higgs bundle $(E, \theta, \delta)$. By the same argument as in Proposition \ref{deformationcomplex}, since $u_0$ is isomorphic and $u_1$ is injective whose cokernel has zero-dimensional support and concentrated in degree one, we have an injection
\begin{equation*}
    i: \bbH^1(C^\bullet)\hookrightarrow \bbH^1(C^\bullet_{F}).
\end{equation*}
Note that Serre duality induces a non-degenerate bilinear pairing on $\bbH^1(C^\bullet_F)$ which corresponds to the well-known symplectic form $\omega_F$ on $\cM_F(n,D)$, see \cite{biswas2018symplectic}. We claim that the pairing $\omega_\Delta$ is obtained by restricting $\omega_F$ to $\bbH^1(C^\bullet) \subset \bbH^1(C^\bullet_F)$. In other words, the corresponding 2-form on $\overline{\cM}^\Delta(n,D)$ is obtained by pulling back the symplectic form $\omega_F$ on $\cM_F(n,D)$. It then follows that $\omega_F$ is closed as well. 

Our claim is equivalent to the commutativity of the following diagram: 
\begin{equation*}
    \begin{tikzcd}
    \bbH^1(C_F^\bullet)\arrow[r]& \bbH^1(\check{C}^\bullet_F)^\vee\arrow[rr]& &\bbH^1(C_F^\bullet)^\vee \arrow[d,two heads,"i^\vee"]\\
    \bbH^1(C^\bullet)\arrow[r]\arrow[u,hook,"i"] &\bbH^1(\check{C}^\bullet) ^\vee\arrow[r]\arrow[u] & \bbH^1(C_1^\bullet)^\vee \arrow[r]& \bbH^1(C^\bullet)^\vee
    \end{tikzcd}
\end{equation*}
The left square diagram commutes  by the functoriality of Serre duality. Then it remains to check the commutativity of the right square diagram. This follows from the commutativity of the diagram of complexes:
\begin{equation*}
\begin{tikzcd}
\check{C}^\bullet_F \arrow[d]&& C^\bullet_F\arrow[ll]\\
\check{C}^\bullet & C_1^\bullet\arrow[l] & C^\bullet  \arrow[l]\arrow[u]
\end{tikzcd}
\end{equation*}
Away from $D$, the diagram clearly commutes. Around $D$, we again trivialize the bundles and the diagram looks like 
\begin{equation*}
    \begin{tikzcd}
    \fbox{$\check{C}^\bullet_F: \ft^\vee(-D) \oplus \fq^\vee (-D) \to \mathcal{L}(\ft^\vee \oplus \fq^\vee)  $}\arrow[d]&\fbox{$C^\bullet_F: \ft(-D)\oplus \fq(-D) \to \mathcal{L}(\ft\oplus \fq) $}\arrow[l] \\
    \fbox{$\check{C}^\bullet: \ft^\vee(-D) \oplus \fq^\vee  \to \mathcal{L}(\ft^\vee \oplus \fq^\vee)  $}&  \fbox{$C^\bullet:\ft(-D)\oplus \fq(-D) \to \mathcal{L}(\ft\oplus \fq(-D)) $} \arrow[d]\arrow[u]\\
    &\fbox{$C^\bullet_1: \ft(-D)\oplus \fq \to \mathcal{L}(\ft\oplus \fq) \arrow[lu] $}
    \end{tikzcd}
\end{equation*}
where we denote by $\mathcal{L}$ the operator sending a bundle $M\mapsto M\otimes K_{\sm}(D)$.

\end{proof}

\begin{proposition}\label{dimension}\label{dimension_count}
\quad 
\begin{enumerate}
    \item $\bbH^0 (C^\bullet)=\bbH^2(C^\bullet)=0$. In particular, the deformations of a diagonally framed Higgs bundle $(E,\theta,\delta)$ are unobstructed. 
    \item $\dim(\bbH^1(C^\bullet)) = (n^2-1) (2g-2 + d)  + (n-1)d .$ 
\end{enumerate}
\end{proposition}
\begin{proof}
(1) Since morphisms between diagonally framed Higgs bundles are in particular morphisms between framed Higgs bundles, automorphisms of diagonally framed Higgs bundles are the same as automorphisms as framed Higgs bundles. So Corollary \ref{no_automorphism} implies that the diagonally framed Higgs bundles are rigid. Hence, $\bbH^0(C^\bullet)=0.$

On the other hand, again by Serre duality, 
\begin{equation*}
    \bbH^2(C^\bullet) \cong (\bbH^0(\check{C}^\bullet))^\vee \cong (\bbH^0(C_1^\bullet) )^\vee
\end{equation*}
where the second isomorphism comes from the isomorphism of the complex \eqref{isom_of_complex}. Finally, recall that from the long exact sequence above, we have that $\bbH^0(C^\bullet) \cong \bbH^0(C_1^\bullet)$ which vanishes as we just proved, hence $\bbH^2(C^\bullet) = 0.$ 

(2) By the definition of $\fg_E'$, we have a short exact sequence 
\begin{center}
    $0\to \fg_E \otimes \cO_{\sm}(-D) \to \fg_E'\to i_*\ft\to 0$
\end{center}
and thus
\begin{align*}
    \chi (\fg_E'\otimes K_{\sm}(D)) &= \chi(\fg_E') + (n^2-1) \deg(K_{\sm}(D))\\
    &= \chi(\ft\otimes \cO_D) + \chi(\fg_E(-D))+ (n^2-1) \deg(K_{\sm}(D))\\
    &= (n-1)d + \chi (\fg_E) + (n^2-1)\deg(\cO_{\sm}(-D))+ (n^2-1) \deg(K_{\sm}(D))\\
    &=  \chi (\fg_E) + (n-1)d+ (n^2-1) (2g-2).
\end{align*}
By (1), $\chi(C^\bullet)  = \bbH^1(C^\bullet)$, so 
\begin{align*}
    \bbH^1(C^\bullet) &= \chi (\fg_E' \otimes K_{\sm}(D)) - \chi(\fg_E(-D))\\
    &=  \chi (\fg_E) + (n-1)d+ (n^2-1) (2g-2)  - \chi(\fg_E) +(n^2-1)d\\
    &= (n^2-1)(2g-2+d) + (n-1)d .
\end{align*}
\end{proof}

\begin{rem}\label{half_dimension}
A direct computation by applying the Riemann-Roch theorem shows that 
\begin{align*}
    \textrm{dim}(B) = \sum_{i=2}^nh^0 (\sm, (K(D))^{\otimes i})  &=  (2g-2+d)\bigg( \frac{n(n+1)}{2}-1\bigg) +  (n-1)(1-g) \\
    &= \frac{1}{2}\big((n^2-1) (2g-2+d) + (n-1)d\big)\\
    &= \frac{1}{2}\textrm{dim}(\bbH^1(C^\bullet)).
\end{align*}
\end{rem}

\begin{proposition}\label{prop:symplectic}
The open subset $\overline{\cM}^{\Delta}(n,D)^{ur}$ of the moduli space $\overline{\cM}^{\Delta}(n,D)$ is a smooth quasi-projective variety of dimension $(n^2-1)(2g-2+d) +(n-1)d$. The tangent space $T_{[(E,\theta, \delta)]}\overline{\cM}^{\Delta}(n,D)^{ur}$ is canonically isomorphic to $\bbH^1(C^\bullet).$ Moreover, $\overline{\cM}^\Delta(n,D)^{ur}$ admits a symplectic form $\omega_\Delta$ which is the restriction of the symplectic form $\omega_F$ on $\cM_F(n,D).$
\end{proposition}
\begin{proof}
All the claims follow immediately from Proposition \ref{tangentspace}, \ref{symplecticform} and \ref{dimension}. The argument to show that $\omega_\Delta$ is a restriction of $\omega_F$ is contained in the proof of Proposition \ref{symplecticform}.
\end{proof}

\begin{proposition}\label{lagrangian}
The fiber of the map $\overline{h}_\Delta: \overline{\cM}^{\Delta}(n,D)^{ur} \to B^{ur}$ is Lagrangian with respect to $\omega_\Delta$. 
\end{proposition}
\begin{proof}
Denote by $(h_1,...,h_l):=h\circ f_1:\cM_F(n,D)\to\cM(n,D)\to \C^l= B$ the composition of the forgetful map and the Hitchin map. According to \cite[Theorem 5.1]{biswas2018symplectic}, the functions $h_i$ Poisson-commute. Since the symplectic form $\omega_\Delta$ on $\overline{\cM}^\Delta(n,D)^{ur}$ is the restriction of the symplectic form $\omega_F$ on $\cM_F(n,D)$, the functions $h_i$ Poisson-commute as well when restricted to $\overline{\cM}^\Delta(n,D)^{ur}$. 

Since the dimension of the fiber $h^{-1}_\Delta(b)$ for $b\in B^{ur}$ is exactly $\frac{1}{2}\textrm{dim}(\overline{\cM}^\Delta(n,D)^{ur})$ by Remark (\ref{half_dimension}), it suffices to show that $\omega_\Delta$ restricted to $\overline{h}^{-1}_\Delta(b)$ vanishes to prove our claim. This follows from Poisson-commutativity of $(h_i)|_{\overline{\cM}^\Delta(n, D)^{ur}}$.

\end{proof}

\begin{proposition}
The tangent space $T_{[(E,\theta, \delta)]}\cM^{\Delta}(n,D)^{ur}$ is canonically isomorphic to $\bbH^1(C^\bullet)$. Moreover, the symplectic form $\omega_\Delta$ on $\overline{\cM}^{\Delta}(n,D)^{ur}$ is invariant under the $S_n^{|D|}$-action. In particular, $\omega_{\Delta}$ descends to a symplectic form $\omega'_\Delta$ on $\cM^{\Delta}(n,D)^{ur}.$ 
\end{proposition}
\begin{proof}
In the proof of Proposition $\ref{tangentspace}$, given an infinitesimal deformation $(E_\e, \theta_\e, \delta_\e)$, the assignment of a 1-cocyle $(\dot{f}_{\alpha\beta}, \dot{\varphi}_\alpha)$ in $\bbH^1(C^\bullet)$ is independent of the reordering of components. That means we have the following commutative diagram
\begin{equation*}
    \begin{tikzcd}
    T_{[(E,\theta, \delta)]}\overline{\cM}^{\Delta}(n,D)^{ur}\arrow[r,"\sim"]\arrow[d,"d\underline{\sigma}"]&\bbH^1(C^\bullet)\arrow[d,"="]\\
    T_{[(E,\theta, \underline{\sigma}\cdot \delta)]}\overline{\cM}^{\Delta}(n,D)^{ur}\arrow[r,"\sim"]&\bbH^1(C^\bullet)
    \end{tikzcd}
\end{equation*}
    for $\underline{\sigma}\in S_n^{|D|}$. The differential of the quotient map 
    \begin{center}
        $dq: T_{[(E,\theta, \delta)]}\overline{\cM}^{\Delta}(n,D)^{ur} \to T_{[(E,\theta, S_n^{|D|}\cdot \delta)]}\cM^{\Delta}(n,D)^{ur} $
    \end{center}
     is an isomorphism. Hence, the canonical identification $ T_{[(E,\theta, \delta)]}\overline{\cM}^{\Delta}(n,D)^{ur} \cong \bbH^1(C^\bullet)$ descends to the tangent space $T_{[(E,\theta, S_n^{|D|}\cdot \delta)]}\cM^{\Delta}(n,D)^{ur}$ via $dq$ and yields a canonical isomorphism $T_{[(E,\theta, S_n^{|D|}\cdot \delta)]}\cM^{\Delta}(n,D)^{ur}\cong \bbH^1(C^\bullet).$ Since the group action of $S_n^{|D|}$ is trivial on $\bbH^1(C^\bullet)$, the symplectic form $\omega_\Delta$ on $\overline{\cM}^{\Delta}(n,D)^{ur}$ is invariant under $S_n^{|D|}$. 
\end{proof}
    
    \begin{corollary}\label{hitchin fibration is S-PIS.}
        The map $h_{\Delta}^{ur}: \cM^\Delta(n,D) ^{ur}\to B^{ur} $ forms a semi-polarized integrable system. 
    \end{corollary}
    \begin{proof}
         By the spectral correspondence proved in Proposition \ref{spectralcorrespondence}, the fibers are semi-abelian varieties.  
        Since $\omega'_\Delta$ descends from the symplectic form $\omega_{\Delta}$, it follows immediately from Proposition \ref{lagrangian} that the fiber of the map $h_\Delta^{ur}: \cM^{\Delta}(n,D)^{ur}\to B^{ur}$ is Lagrangian with respect to $\omega'_\Delta.$    
    \end{proof}
    \begin{rem}
    For a fixed $b\in B^{ur}$, the fiber $(h_{\Delta}^{ur})^{-1}(b)$ is a semi-abelian variety $\Prym(\ts^\circ_b, \sm^\circ)$ which admits a $(\C^*)^{(n-1)d}$-action. This group action can be seen by viewing $\Prym(\ts^\circ_b,\sm^\circ)$ as parametrizing framed line bundles on $\ts_b$ which correspond to unordered diagonally framed $SL(n,\C)$-Higgs bundles under spectral correspondence. Then $(\C^*)^{(n-1)d}$ acts simply transitively on the space of framings over $D$ for each fixed line bundle, and the quotient map is equivalent to the natural map $\Prym(\ts^\circ_b,\sm^\circ) \to \Prym(\ts_b,\sm)$ of forgetting the framings. Applying this fiberwise quotient by $(\C^*)^{(n-1)d}$ to the fibration $\cM^{\Delta}(n,D)^{ur}\to B^{ur}$, we see that the quotient map is precisely the forgetful map $f_1: \cM^{\Delta}(n,D)^{ur} \to \cM(n,D)^{ur}$. Thus, this provides a geometric interpretation of the fact that the Poisson integrable system $\cM(n,D)^{ur}\to B^{ur}$ is realized as the fiberwise compact quotient of the semi-polarized integrable system $\cM^{\Delta}(n,D)\to B^{ur}$ as discussed in Section \ref{sec: semipolarized IS}. 
    \end{rem}
    
\subsection{Cameral description}\label{sec: cameral description}
Although the spectral curve description is more intuitive and straightforward, it only works for classical groups. To describe the general fiber of  Hitchin system for any reductive group $G$ as well as prove DDP-type results, it is more natural to use the cameral curve description and \textit{generalized Prym variety}. In this section, we focus on the extension of classical results in our case (A-type). The goal is to describe the mixed Hodge structure of unordered diagonally framed Hitchin fibers in terms of cameral curves. We refer to \cite{DonagiGaitsgory}\cite{Langlandsduality} for more basics and details about the cameral description. 

In this section, we use general notation from algebraic group theory with an eye towards a  generalization of the previous arguments to any complex reductive group $G$ (see Remark \ref{rem:forthcoming paper}). We fix a maximal tori $T \subset G$ and denote by $\ft$ the corresponding Lie algebra. We also write $W$ for the Weyl group. 

As the Hitchin base $B$ can be considered as the space of sections of $K_\Sigma(D) \otimes \ft/W$, we have the following commutative diagram
\begin{equation}\label{eq:Cameralgluing}
    \begin{tikzcd}
    \widetilde{\bm{\Sigma}} \arrow[r] \arrow[d,"\bm{\widetilde{p}}"] & \bm{\widetilde{U}} :=\Tot(K_\Sigma(D) \otimes  \ft) \arrow[d,"\bm{\phi}" ] \\
    \Sigma \times B \arrow[r, "ev"] & \bm{U} :=\Tot(K_\Sigma(D) \otimes \ft/W)
    \end{tikzcd}
\end{equation}
where $\widetilde{\bm{\Sigma}}$ is the \textit{universal cameral curve} of $\Sigma$. By composing $\bm{\widetilde{p}}$ with the projection to $B$, we have a family of cameral curves $p:\widetilde{\bm{\Sigma}} \to B$ whose fiber is a $W$-Galois cover of the base curve $\Sigma$. An interesting observation is that in the meromorphic case, one can consider the universal cameral pair $(\widetilde{\bm{\Sigma}}, \widetilde{\bm{D}}:=\bm{\widetilde{p}}^{-1}(D \times B))$, which allows us to extend the notion of generalized Prym variety \cite{donagi1993decomposition}. Let's recall the definition of the generalized Prym variety. For a generic $b \in B$, we define a sheaf of abelian groups $\cT_b$ by 
\begin{equation*}
    \cT_b(U):=\{t \in \widetilde{p}_{b*}(\Lambda_G \otimes \cO_{\widetilde{\Sigma}}^*)^W(U)|\, \alpha(t)|_{M^\alpha}=+1 \quad \forall \alpha \in R(G) \}
\end{equation*}
where $R(G)$ is a root system and $\Lambda_G$ is the cocharater lattice and $M^\alpha$ is the ramification locus of $\widetilde{p}_b:\widetilde{\Sigma}_b \to \Sigma$ fixed by the reflection $S_2 \in W$ corresponding to $\alpha$. We define the generalized Prym variety of $\widetilde{\Sigma}_b$ over $\Sigma$ as the sheaf cohomology $H^1(\Sigma, \cT_b)$. 
\begin{theorem}[\cite{DonagiGaitsgory}\cite{HHP2010projective}]
    For $b \in B^\circ$, the fiber $h^{-1}(b)$ in the meromorphic Hitchin system is isomorphic to the generalized Prym variety $H^1(\Sigma, \cT_b)$: 
    \begin{equation*}
        h^{-1}(b) \cong H^1(\Sigma, \cT_b)
    \end{equation*}
where $B^\circ$ is the locus of smooth cameral curves with simple ramifications.     
\end{theorem}

Let $i_D:D \hookrightarrow \Sigma \hookleftarrow \Sigma \setminus D:j_D$ be inclusions. Associated to the cameral pair $(\widetilde{\Sigma}_b, \widetilde{D}_b)$, one can extend the generalized Prym variety to $H^1(\Sigma, j_{D!}j^*_D\cT_b)$ which is isomorphic to $h_\Delta^{-1}(b)$ when $G=SL(n,\C)$.

\begin{proposition}\label{prop:cameraldescription}
 For $b \in B^{ur}$, the unordered diagonally framed Hitchin fiber $(h_{\Delta})^{-1}(b)$ is isomorphic to $H^1(\Sigma, j_{D!}j_D^*\mathcal{T}_b)$. In particular, it is a semi-abelian variety which corresponds to the $\Z$-mixed Hodge structure
\begin{equation} \label{eq:relative mixed Hodge}
    (H^1(\Sigma,D,(\widetilde{p}_{b*}\Lambda_{SL(n)})^W)_{\textrm{tf}}, H^1(\widetilde{\Sigma}_b, \widetilde{D}_b, \ft)^W)
\end{equation}
whose weight and Hodge filtration are induced from Hodge structure of $H^1(\widetilde{\Sigma}_b, \ft)^W$ and $H^0(\widetilde{D}_b, \ft)^W$.
\end{proposition}

\begin{proof}
In the case $G=SL(n,\C)$, it is known that $\cT_b$ is isomorphic to $(\widetilde{p}_{b*}(\Lambda_{SL(n)})^W)\otimes \cO^*_\Sigma$ (see \cite[Section 3]{Langlandsduality}). In other words, there is an isomorphism between the cocharacter lattice of $H^1(\Sigma, \cT_b)$ and $H^1(\Sigma, (\widetilde{p}_{b*}\Lambda_{SL(n)})^W\otimes \cO_{\sm}^*)$. We can then extend this result to show that the generalized Prym variety $H^1(\Sigma,j_{D!}j^*_D\cT_b)$ is isomorphic to the Jacobian of $\Z$-mixed Hodge structure (\ref{eq:relative mixed Hodge}) on $H^1(\Sigma,D,(\widetilde{p}_{b*}\Lambda_{SL(n)})^W)_\textrm{tf}$. 

On the other hand, the fiber $(h_{\Delta})^{-1}(b)$ is isomorphic to the Jacobian of the $\Z$-mixed Hodge structure on $H^1(\Sigma, D,\cK_b)_\textrm{tf}$ where $\cK_b:=\ker(\Tr:\bar{p}_{b*}\Z \to \Z)$ (see Lemma \ref{rem:Prym sheaf}). To relate it with the cameral description, we consider an isomorphism of sheaves, 
\begin{equation}
    (\widetilde{p}_{b*}\Lambda_{SL(n)})^W \cong \cK_b
\end{equation}
which will be proved in Lemma \ref{lem:spectraltocameral}. It induces the isomorphism of $\Z$-mixed Hodge structures of type $\{((0,0), (0,1) , (1,0)\}$ on the torsion free part of the relative sheaf cohomologies:
\begin{equation*}
    H^1(\Sigma,D,(\widetilde{p}_{b*}\Lambda_{SL(n)})^W)_\textrm{tf} \cong H^1(\Sigma,D,\cK_b)_\textrm{tf}.
\end{equation*}
\end{proof}

\begin{rem}\label{rem:forthcoming paper} 
In \cite[Section 7]{biswas2019symplecticGHiggs}, the theory of diagonally\footnote{In \cite{biswas2019symplecticGHiggs}, the authors use the word "relatively" instead of "diagonally".} framed Higgs bundles for arbitrary reductive group $G$ and its abelianization have been studied. Similar to the case of $SL(n,\C)$, a generic fiber of diagonally framed Hitchin fiber $(h_\Delta)^{-1}(b)$ is a semi-abelian variety, which is an extension of a meromorphic Hitchin fiber $h^{-1}(b)$ by the affine torus $T^{|D|}/Z(G)$ where $Z(G)$ is the center of the group $G$ (\cite[Corollary 7.10]{biswas2019symplecticGHiggs}. We explain how the relative version of the generalized Prym variety $H^1(\Sigma, j_{D!}j^*_D\cT_b)$ gives this description as follows.
\\ 
Note that an additional data of diagonal framing amounts to specifying $W$-equivariant section of $T$-bundle at $D$. This can be formulated as $H^0(D_b, \mathcal{T}_b)=H^0(D, (\widetilde{p}_{b*}\Lambda_{G})^W \otimes \C^*)$ modulo 
the action of the center $Z(G)$. Moreover, the distinguished triangle in the constructible derived category of $\Sigma$, $D_c^b(\Sigma)$
\begin{equation*}
    j_{b!}j^*_b \to id \to {i_{b}}_* i_b^* \xrightarrow{+1}
\end{equation*}
induces the long exact sequence as follows
\begin{equation}
    H^0(\Sigma, j_{b!}j^*_b\mathcal{T}_b) \to H^0(\Sigma, \mathcal{T}_b) \xrightarrow{i_D^*} H^0(D, \mathcal{T}_b) \to H^1(\Sigma, j_{b!}j^*_b\mathcal{T}_b) \to H^1( \Sigma, \mathcal{T}_b) \to 0.
\end{equation}
Here, $H^0(\Sigma, \mathcal{T}_b)$ is the space of $W$-equivariant maps, $\Hom_W(\widetilde{\Sigma}_b, T)$, which takes values 1 on $M_{\widetilde{\Sigma}_b}^\alpha$ for every root $\alpha$. Note that 
\begin{equation*}
    Z(G)=\{t \in T^W|\alpha(t)=1 \quad \text{for all} \quad \alpha \in R(G) \}.
\end{equation*}
Therefore, the cokernel of $i_D^*:H^0(\Sigma, \mathcal{T}_b) \to H^0(D, \mathcal{T}_b)$ can be identified with $T^{|D|}/Z(G)$, a level subgroup. Clearly this is a copy of $\C^*$'s, so we have the semi-abelian variety $H^1(\Sigma,j_{b!}j^*_b\mathcal{T}_b)$ as an extension of $H^1( \Sigma, \mathcal{T}_b)$ by $T^{|D|}/Z(G)$.

\end{rem}
\subsection{Abstract Seiberg-Witten differential}
Using the cameral description introduced in Section \ref{sec: cameral description}, one can define an abstract Seiberg-Witten differential. For simplicity, we denote $\mathsf{V}=(V_\Z, W_\bullet V_\Z, F^\bullet V_\cO)$ the variation of $\Z$-mixed Hodge structure over $B$ introduced in Proposition \ref{prop:cameraldescription}. Consider the dual variation of $\Z$-mixed Hodge structure, $\mathsf{V}^\vee=(V^\vee_\Z, W_\bullet V^\vee_\Z, F^\bullet V^\vee_\cO)$ whose fiber at $b \in B$ is the $\Z$-mixed Hodge structure $(H^1(\Sigma \setminus D, (\widetilde{p}_{b,*}\Lambda^\vee_{SL(n)})^W), H^1(\widetilde{\Sigma}_b \setminus \widetilde{D}_b, \ft)^W)$. 

Recall that in the case of holomorphic Hitchin system \cite{HHP2010projective}, the Seiberg-Witten differential is a holomorphic one-form which is obtained by the tautological section of the pullback of $K_\Sigma$ under $\Tot(K_\Sigma) \to \sm$. Similarly, in the meromorphic case, the tautological section of the pullback of $K_\Sigma(D)$ under $\Tot(K_\Sigma(D))\to \sm$ gives the logarithmic 1-form $\theta$. In other words, for each $b \in B$, we have
\begin{equation*}
   \theta|_{\widetilde{\Sigma}_b} \in H^0(\widetilde{\Sigma}_b, \ft\otimes \Omega^1_{\widetilde{\Sigma}_b}(\log \widetilde{D}_b))^W =F^1H^1(\widetilde{\Sigma}_b \setminus \widetilde{D}_b, \ft)^W
\end{equation*}
This is the natural candidate of the abstract Seiberg-Witten differential, which we will denote by $\lambda_\Delta$. 

\begin{lemma}\label{lem:hitchin base identification}\cite{hurtubisemarkman98} \cite{Kjiri2000GeneralizedHitchin}
For each $b \in B^{ur}$, the logarithmic $2$-form $d\theta_b$ induces an isomorphism by the contraction map
\begin{equation*}
\begin{aligned}
    T_bB^{ur} & \xrightarrow{\cong} H^0(\widetilde{\Sigma}_b, \ft\otimes \Omega^1_{\widetilde{\Sigma}_b}(\log \widetilde{D}_b))^W \\
    \mu & \mapsto \iota_\mu d\theta_b
\end{aligned}
\end{equation*}
\end{lemma}
\begin{proof}
From the cameral description, the tangent space $T_bB^{ur}$ is isomorphic to the space of $W$-invariant deformation of $\cs_b$ in $\Tot(\ft \otimes K_\sm(D))$. The latter space is given by $H^0(\cs_b, N_{\cs_b})^W$ where $N$ is the normal bundle of the embedding $\sm_b \hookrightarrow \Tot(\ft \otimes K_\sm(D))$. The logarithmic two form $d\theta$ on $\Tot(\ft \otimes K_\sm(D))$ induces a sheaf homomorphism $N_{\cs_b} \to \ft \otimes \Omega^1_{\cs_b}(\log \widetilde{D_b})$ whose $W$-invariant global sections are isomorphic \cite[Section IV]{Kjiri2000GeneralizedHitchin}. Under the identification $T_bB^{ur} \cong H^0(\cs_b, N_{\cs_b})^W$, the induced isomorphism of $W$-invariant global sections is given by the contraction. 
\end{proof}

\begin{proposition}\label{prop:abstractseiberghitchin}
By applying the Gauss-Manin connection to $\lambda_\Delta$, one can obtain an isomorphism of vector bundles over $B^{ur}$ 
\begin{align*}
    \phi_{\lambda_\Delta}: TB^{ur} & \xrightarrow{\cong} F^1V_\cO^\vee \\
    \mu & \mapsto \nabla^{GM}_{\mu}(\lambda_{\Delta})
\end{align*}
In other words, $\lambda_\Delta \in H^0(B^{ur}, V^\vee_\cO)$ is the abstract Seiberg-Witten differential.
\end{proposition}

For the rest of this section, we will write $B^{ur}$ simply as $B$. Note that having a variation of $\Z$-mixed Hodge structures over $B$ corresponds to having the classifying map to mixed period domain; $\Phi:B \to \cD/\Gamma$. It admits a holomorphic lift \cite{mixedperiod} $\widetilde{\Phi}:B \to \cD$ which factors through relative Kodaira-Spencer map $\kappa:T_{B,b} \to H^1(\widetilde{\Sigma}_b, T_{\widetilde{\Sigma}_b}(-\log \widetilde{D}_b))$ 
\begin{equation}\label{diagram:mixed period domain}
    \begin{tikzcd}
        T_{B,b} \arrow[rr, "d\widetilde{\Phi}"] \arrow[rd, "\kappa"] && T_{\cD, \widetilde{\Phi}(b)} \\
        & H^1(\widetilde{\Sigma}_b, T_{\widetilde{\Sigma}_b}(-\log\widetilde{D}_b)) \arrow[ru, "m^\vee"] 
    \end{tikzcd}
\end{equation}
where $m^\vee:H^0(\widetilde{\Sigma}_b, \ft \otimes \Omega^1_{\widetilde{\Sigma}_b}(\log \widetilde{D}_b))^W \otimes H^1(\widetilde{\Sigma}_b, T_{\widetilde{\Sigma}_b}(-\log\widetilde{D}_b)) \to H^1(\widetilde{\Sigma}_b, \ft \otimes\cO_{\widetilde{\sm}_b})^W$ is the logarithmic contraction. With this description, one can compute the Gauss-Manin connection of $\lambda_\Delta$.

\begin{proof}
This is a generalization of the work of Hertling-Hoevenaars-Posthuma \cite[Section 8.2]{HHP2010projective}. We will follow closely the approach in \textit{loc.cit.}. The main reason we can apply the similar arguments is that  we restrict to cameral covers with no ramification over the divisors. 

 Consider the following commutative diagram induced by the universal cameral pairs $(\widetilde{\bm{\Sigma}}, \widetilde{\bm{D}}:=\bm{\widetilde{p}}^{-1}(D \times B))$ introduced in (\ref{eq:Cameralgluing})
 \begin{equation*}
     \begin{tikzcd}
     \widetilde{\bm{\Sigma}}^\circ \arrow[r, hook, "j"] \arrow[d, "\bm{\widetilde{p}}^\circ"'] & \widetilde{\bm{\Sigma}} \arrow[d,"\bm{\widetilde{p}}" ] \arrow[dd, bend left=60, "f"]\\
     \sm^\circ \times B \arrow[r, hook] \arrow[rd, "pr_2"] & \sm \times B \arrow[d, "pr_2"] \\
    & B     
     \end{tikzcd}
 \end{equation*}
 where $\widetilde{\bm{\Sigma}}^\circ:=\widetilde{\bm{\Sigma}} \setminus\widetilde{\bm{D}}$, $f:=pr_2 \circ \bm{\widetilde{p}}:\widetilde{\bm{\Sigma}} \to B$. We also define $f^\circ:=f|_{\widetilde{\bm{\Sigma}}^\circ}$. We take open covers $U$ and $U'$ of $\widetilde{\bm{\Sigma}}$ such that for each $b \in B$, $U_b:= U\cap f^{-1}(b)$ is the complement of the ramification locus of $\Tilde{p_b}:\cs_b \to \Sigma$ and $U'_b:= U'\cap f^{-1}(b)$ is a disjoint union of small discs around the branch locus. In particular, the divisor $\bm{\widetilde{D}}$ does not intersect $U\cap U'$.

For simplicity, we will write $f_*$ and $f^\circ_*$ as the $W$-equivariant direct image functors $(f_*(-))^W$ and $(f^\circ_*(-))^W$ respectively. Then we can argue as in \cite[Section 2]{usui1983} that 
\begin{align*}
    V_{\cO}^\vee &= (R^1f_*^\circ \ft)\otimes_\C \cO_B
    \cong R^1f_*^\circ (\ft\otimes_\C (f^\circ)^{-1}\cO_B) \\
    &\cong \bbR^1f_*^\circ \left(\ft\otimes_\C \Omega^\bullet_{\bm{\widetilde{\sm}^\circ}/B}\right)\cong \bbR^1 f_*j_*\left(\ft\otimes_\C \Omega^\bullet_{\bm{\widetilde{\sm}^\circ}/B}\right)
    \cong \bbR^1f_*\left(\ft \otimes_\C \Omega^\bullet_{\widetilde{\bm{\Sigma}}/B}\left(\log\bm{\widetilde{D}}\right)\right)
\end{align*}
where $\ft\otimes_\C \Omega^\bullet_{\widetilde{\bm{\Sigma}}/B}\left(\log\bm{\widetilde{D}}\right)$ is the the ($\ft$-valued) relative logarithmic de Rham complex whose definition we refer to \cite[Definition 2.1]{usui1983} and we will simply denote the complex by $\widetilde{\Omega}^\bullet_{\widetilde{\bm{\Sigma}}/B}$.

Now we consider the \u{C}ech resolution of  $f_*\widetilde{\Omega}^\bullet_{\widetilde{\bm{\Sigma}}/B}$:
\begin{equation*}
    \begin{tikzcd}
    f_*\widetilde{\Omega}^1_{\widetilde{\bm{\Sigma}}/B} \arrow[r] & f_{U*}\widetilde{\Omega}^1_{\widetilde{\bm{\Sigma}}/B} \oplus f_{U'*}\widetilde{\Omega}^1_{\widetilde{\bm{\Sigma}}/B} \arrow[r, "\delta"] & f_{U\cap U'*}\widetilde{\Omega}^1_{\widetilde{\bm{\Sigma}}/B} \\
    f_*\widetilde{\Omega}^0_{\widetilde{\bm{\Sigma}}/B} \arrow[r] \arrow[u]& f_{U*}\widetilde{\Omega}^0_{\widetilde{\bm{\Sigma}}/B} \oplus f_{U'*}\widetilde{\Omega}^0_{\widetilde{\bm{\Sigma}}/B} \arrow[r, "\delta"] \arrow[u, "d"]& f_{U\cap U'*}\widetilde{\Omega}^0_{\widetilde{\bm{\Sigma}}/B} \arrow[u, "d"]
    \end{tikzcd}
\end{equation*}
where we denote by $f_U:= f|_U, f_U':= f|_U', f_{U\cap U'}:= f|_{U\cap U'}$. It allows us to describe global sections in $\bbR^1f_*\widetilde{\Omega}^\bullet_{\widetilde{\bm{\Sigma}}/B}$ as the elements 
\begin{equation*}
    (\alpha_U, \alpha_{U'}, g_\alpha) \in (\widetilde{\Omega}^1_{\widetilde{\bm{\Sigma}}/B}(U) \oplus \widetilde{\Omega}^1_{\widetilde{\bm{\Sigma}}/B}(U')) \bigoplus \widetilde{\Omega}^0_{\widetilde{\bm{\Sigma}}/B}(U \cap U')
\end{equation*}
such that $(\alpha_U-\alpha_{U'})|_{U \cap U'}=d(g_\alpha)$.

Now, take $\lambda_{\Delta} \in F^1V_{\cO}^\vee$ which is represented as an element $\lambda_{\Delta}=(\lambda_{\Delta U}, \lambda_{\Delta U'}, 0)$. Choose a holomorphic vector field $\mu$ on $B$ with local liftings $\mu_U$ and $\mu_{U'}$. Then the Gauss-Manin connection $\nabla^{GM}_\mu(\lambda_{\Delta})$ is given by 
\begin{equation*}
    \nabla_\mu(\lambda_{\Delta U}, \lambda_{\Delta U'}, 0)=(\mathfrak{L}_{\mu_U}\lambda_{\Delta U}, \mathfrak{L}_{\mu_{U'}}\lambda_{\Delta U'}, \iota_{\mu_U-\mu_{U'}}\lambda_{\Delta U'})
\end{equation*}
where $\iota$ is contraction map and $\mathfrak{L}$ is a Lie-derivative operator. In particular, the last term is given by $m^\vee(\lambda_{\Delta},\kappa(\mu))$ in (\ref{diagram:mixed period domain}). Since $\iota_{\mu_U-\mu_{U'}}\lambda_{\Delta U'}=\delta(\iota_{\mu_U}\lambda_{{\Delta}U}, \iota_{\mu_{U'}}\lambda_{\Delta U'})$, we conclude that $\nabla^{GM}_\mu(\lambda_{\Delta}) \in F^1V_{\cO}^\vee$ and $\phi_{\lambda_{\Delta}}$ is well-defined. Moreover, Cartan's formula implies that $\nabla^{GM}_\mu(\lambda_{\Delta})$ is cohomologous to $(\iota_{\mu_U}d\lambda_{\Delta}|_U, \iota_{\mu_{U'}}d\lambda_{\Delta}|_{U'}, 0)$, which represents the element $\iota_\mu d\theta$. The conclusion follows from Lemma \ref{lem:hitchin base identification}. 
\end{proof}

\section{Calabi-Yau integrable systems}

\subsection{Construction}
In this section, we shall generalize Smith's elementary modification idea \cite{smith2015quiver} to construct a (semi-polarized) Calabi-Yau integrable system. A similar construction is also used in the work of Abrikosov \cite[Section 6]{abrikosov2018potentials} and Smith \cite{smith2020floer}.
 
First, we describe the construction of a family of Calabi-Yau threefolds. Let $V:= \Tot(K_{\sm}(D)\oplus (K_{\sm}(D))^{n-1}\oplus K_{\sm}(D)) $ and consider the short exact sequence
        \begin{equation*}
            0\to \cO_{\sm}(-D)\xrightarrow{\alpha} \cO_{\sm}\to i_{D*}\cO_D \to 0.
        \end{equation*}
        where we recall $i_D:D \hookrightarrow X$ is the natural inclusion. Suppose $u$ is a local frame of $\cO_{\sm}(-D)$. In terms of a local coordinate $z$ around a point of $D$ where $z=0$, $\alpha(u)$ is represented by $f\cdot u$ where $f$ is a locally defined function that vanishes at $z=0.$ 
        We define an elementary modification $\wv$ of $V$ along the first component: 
        \begin{equation*}
             \wv :=\Tot (K_{\sm}(D-D)\oplus (K_{\sm}(D))^{n-1}\oplus K_{\sm}(D))\to \Tot(K_{\sm}(D)\oplus (K_{\sm}(D))^{n-1}\oplus K_{\sm}(D)) 
        \end{equation*}
        and denote the projection map by $\wpi: \wv\to \sm.$ 
        
        For $b =(b_2(z),...,b_n(z)) \in B= \oplus_{i=2}^n H^0({\sm}, K_{\sm}(D)^{\otimes i})$, we define the threefold $X_b$ as the zero locus of a section in $\Gamma(\wv, \wpi^*K_{\sm}(D)^{\otimes n})$:
        \begin{equation} \label{defining equation}
            X_b:= \{\alpha(x)y - s^n  -\wpi^*b_2(z) s^{n-2} -... -\wpi^*b_n(z) =0 \}\subset \wv
        \end{equation}
        with the projection $\pi_b:X_b \to \Sigma$. Here we denote by $x, y$ and $s$ the tautological sections of the pullback of $K_\sm,(K_\sm(D))^{n-1}$ and $K_\sm(D)$, respectively.
        Note that each term in the equation (\ref{defining equation}) is a section of $\wpi^*K_{\sm}(D)^{\otimes n}$. More explicitly, we have
        \begin{align*}
            &x\in \Gamma ( \wv, \wpi^*K_{\sm}), \quad \alpha(x) \in \Gamma (\wv, \wpi^*K_{\sm}(D)), \quad 
             y\in \Gamma ( \wv, \wpi^*(K_{\sm}(D))^{n-1})\\ &
            s\in \Gamma ( \wv, \wpi^*K_{\sm}(D)),\quad  \pi^*b_i \in \Gamma (\wv,\wpi^*(K_{\sm}(D))^i) 
        \end{align*}
        This construction gives rise to a family of quasi-projective threefolds $\pi:\cX \to B$ with the following commutative diagram
        \begin{equation*}
        \begin{tikzcd}
        \cX \arrow[d, "\bm{\pi}"'] \arrow[rd, "\pi"]& \\
        \sm \times B \arrow[r, "pr_2"]& B
        \end{tikzcd}
        \end{equation*}
        where $pr_2:\sm \times B \to B$ is the natural projection and $\pi:=pr_2 \circ \bm{\pi}$. Next, we show that the threefold $X_b$ is indeed a non-singular Calabi-Yau threefold.
        \begin{proposition}
        The threefold $X_b$ has trivial canonical bundle.
        \end{proposition}
        \begin{proof}
           By the adjunction formula, 
        \begin{equation*}
            K_{X_b} = K_{\wv} \otimes \wpi^* (K_{\sm}(D))^{\otimes n}|_{X_b}.
        \end{equation*}
        where $\wpi: \wv\to {\sm}.$ Note that 
        \begin{equation*}
            K_{\wv} = \wpi^*\det ( \wv^\vee ) \otimes \wpi^*K_{\sm} \cong \wpi^*(K_{\sm}^{-n-1}(-nD)) \otimes \wpi^*K_{\sm} \cong \wpi^*(K_{\sm}^{-n}(-nD)).
        \end{equation*}
        So it follows that 
        \begin{equation*}
            K_{X_b} = \wpi^*(K_{\sm}^{-n}(-nD)) \otimes \wpi^* (K_{\sm}(D))^{\otimes n}|_{X_b} \cong \cO_{X_b}.
        \end{equation*}
        \end{proof}
    \begin{proposition}
       For each $b\in B^{ur}$, the threefold $X_b$ is non-singular. 
    \end{proposition}
    \begin{proof}
            This is a local statement, so we can restrict to neighbourhoods in $\sm.$ 
            Around a point of $D$ with local coordinate $z$, the local model of $X_b$ is 
            \begin{equation*}
                \{f(z) xy - s^n  -\widetilde{b}_2(z) s^{n-2} -... -\widetilde{b}_n(z) =0 \}\subset \C^3_{(x,y,s)} \times \C_z,
            \end{equation*}
            where $\widetilde{b}_i$ are now functions of $z$, and $f(z)$ is function with zero only at $z=0.$ We check smoothness by examining the Jacobian criterion. The equation
            \begin{equation}\label{spectral_cover_equation}
                \frac{\partial}{\partial s} \big(s^n  -\widetilde{b}_2(z) s^{n-2} -... -\widetilde{b}_n(z) \big ) = 0
            \end{equation}
            implies that, for each $z$, the equation $s^n  -\widetilde{b}_2(z) s^{n-2} -... -\widetilde{b}_n(z)=0$ must have repeated solutions, this happens only when $z$ is at the branch locus of the spectral curve associated to $b$. The remaining equations in the Jacobian criterion are
            \begin{equation*}
                f(z) y= 0, \quad f(z)x =0 , \quad f'(z) xy +\frac{\partial}{\partial z}\big(s^n  -\widetilde{b}_2(z) s^{n-2} -... -\widetilde{b}_n(z)\big)=0
            \end{equation*}
            When $x=y=0$, the equation $\frac{\partial}{\partial z}\big(s^n  -\widetilde{b}_2(z) s^{n-2} -... -\widetilde{b}_n(z)\big)=0$ has no solution since we assume that the spectral curve associated to $b$ is smooth and equation (\ref{spectral_cover_equation}) has a solution. Hence, it must be the case that $x\neq 0$ or $y\neq 0$ which implies that $f(z)=0$ or equivalently $z=0$. However, since we assume $b\in B^{ur}$ which means that the spectral curve must be unramified over $D$, this is a contradition and so $X_b$ is non-singular around $D.$ 
            
            Away from $D$, a similar argument shows that the threefold is non-singular over the local neighbourhood. Hence, $X_b$ is non-singular everywhere. 
            
    \end{proof}
    Again, by examining the defining equation (\ref{defining equation}), we can list the types of fibers of the map $\pi_b:X_b \to \sm$: 
    \begin{itemize}
        \item For $p\in D$ with coordinate $z=0$, the fiber is defined by the equation \\ $s^n  -\widetilde{b}_2(z) s^{n-2} -... -\widetilde{b}_n(z)=0$, i.e. disjoint union of $n$ copies of $\C^2.$ 
        \item For a critical value $p$ of $\pi_b$, the fiber is defined by $xy - \prod_{i=1}^m(s-s_i)^{k_i} $ where $\sum_{i=1}^m k_i = n$ $(m<n)$. Hence, the fiber is a singular surface with $A_{k_i-1}$-singularity at $s_i$. 
        \item For $p$ away from $D$ and the discriminant locus of $\pi_b$, the fiber is defined by $(xy -s^n ) -\widetilde{b}_2(z) s^{n-2} -... -\widetilde{b}_n(z)= 0$ and smooth, so it is isomorphic to a smooth fiber of the universal unfolding of $A_{n-1}$-singularity $\C^2/\Z_n$.
    \end{itemize}
Consider the subfamily $\pi^\circ:\cX^\circ \to B$ of the family of quasi-projective Calabi-Yau threefolds $\mathbf{\pi}:{\cX} \to B$, defined via the following commutative diagram
    \begin{equation*}
    \begin{tikzcd}
        (\cX^\circ): = \cX \setminus \bm{\pi}^{-1}(D\times B) \arrow[r, hook] \arrow[rd, "\pi^\circ:=\pi|_{\cX^\circ}"']& \cX \arrow[d, "\pi"] \\
        & B.
    \end{tikzcd}
    \end{equation*}
    whose fiber at $b \in B$ is $X_b^\circ:= X_b \setminus \pi_b^{-1}(D)$. We write $(-)^{ur}$ for the restriction of the family over $B^{ur} \subset B$. For our purposes later, it is useful to study the relation between the family $(\pi^\circ)^{ur}:(\cX^\circ)^{ur}\to B^{ur}$ and the family of Calabi-Yau threefolds in the holomorphic case \cite{diaconescu2006intermediate}\cite{beck2017hitchin} obtained by gluing Slodowy slices. 

Recall that in the classical case \cite{slodowy1980four}, the Slodowy slice $S\subset \fg$ provides a semi-universal $\C^*$-deformation $\sigma:S\to \ft /W$ of simple singularities via the adjoint map $\sigma:\fg\to \ft/W$. However, if we denote by $d_j$ the standard ($\C^*$-action) weights of the generators of the coordinate ring $\C[\chi_1,...,\chi_j]$ of $\ft/W$, then the weights on $\C[\chi_1,...,\chi_j]$ must be chosen as $2d_j$ for $\sigma$ to be $\C^*$-equivariant (see \cite[Remark 2.5.3]{beck2020folding}, \cite{slodowy1980four}). 

Now we choose a theta characteristic $L$ on $\sm$, i.e. $L^2 \cong K_\sm$. Since $L^2|_{\sm^\circ}\cong K_{\sm}|_{\sm^\circ}\cong K_{\sm}(D)|_{\sm^\circ}$, we have an isomorphism of associated bundles over $\sm^\circ$
\begin{equation*}
    L|_{\sm^\circ}\times_{\C^*} \ft/W \cong  K_{\sm}(D)|_{\sm^\circ}\times_{\C^*} \ft/W
\end{equation*}
where the weights of the $\C^*$-action on both sides are different: the left hand side has weights $2d_j$ and the right hand side has weights $d_j.$ As the map $\sigma:S\to \ft/W$ is $\C^*$-equivariant, we can glue it along $\Tot(L)$ to obtain 
\begin{equation*}
    \bm{\sigma} : \bm{S}: = \Tot (L\times _{\C^*}S) \to \Tot(L\times_{\C^*} \ft/W)
\end{equation*}
and its restriction
\begin{equation*}
    \bm{\sigma}|_{\sm^\circ} : \bm{S}|_{\sm^\circ}: = \Tot (L\times _{\C^*}S)|_{\sm^\circ} \to \Tot(L|_{\sm^\circ}\times_{\C^*} \ft/W)\cong \Tot(K_\sm(D)|_{\sm^\circ}\times_{\C^*}\ft/W) = \bm{U}|_{\sm^\circ} .
\end{equation*}
Pulling back under the evaluation map from $\Sigma \times B$, one gets the two families of quasi-projective threefolds $\pi':\cY \to B$ and ${\pi'}^{\circ}:=\pi'|_{\cY^\circ}:\cY^\circ \to B$ as follows:
\begin{equation}\label{CYfamilies}
\begin{tikzcd}
    \cY \arrow[r] \arrow[d, "\bm{\pi'}"] \arrow[dd, bend right=60, "\pi'"'] & \bm{S} \arrow[d, "\bm{\sigma}"] \\
    \Sigma \times B \arrow[r, "ev"] \arrow[d, "pr_2"] & \bm{U} \\
    B &
    \end{tikzcd},\quad 
    \begin{tikzcd}
    (\cY^\circ) \arrow[r] \arrow[d, "\bm{\pi'}|_{\cY^\circ}"] \arrow[dd, bend right=60, "{\pi'}^\circ"'] & \bm{S}|_{\sm^\circ} \arrow[d, "\bm{\sigma}|_{\sm^\circ}"] \\
    \Sigma^\circ \times B \arrow[r, "ev"] \arrow[d, "pr_2"] & \bm{U}|_{\sm^\circ} \\
    B &
    \end{tikzcd}
\end{equation}
We define $\cY^{ur}$ and $(\cY^\circ)^{ur}$ as the restriction over $B^{ur} \subset B$.
\begin{lemma}\label{gluingmodel}
 We have an isomorphism of the two families over $B^{ur}$
 \begin{equation*}
     \begin{tikzcd}
     (\cY^\circ)^{ur} \arrow[rr, "\cong"] \arrow[rd, "({\pi'}^\circ)^{ur}"'] & & (\cX^\circ)^{ur} \arrow[ld, "(\pi^\circ)^{ur}"] \\
     & B^{ur} &
     \end{tikzcd}
 \end{equation*}
In particular, $Y_b^\circ\cong X_b^\circ$ where $Y_b^\circ$ is a member of the family $\cY^\circ$.
\end{lemma}

\begin{proof}

For type $A$, we have a semi-universal $\C^*$-deformation of $A_{n-1}$ singularities (see \cite[Theorem 1]{katz1992gorenstein})  as follows:
\begin{equation}
\begin{aligned}
\sigma': H:=\{xy- s^n - b_2s^{n-2}- ...- b_n = 0\} \subset  \C^3 \times \C^{n-1} &\to \C^{n-1}\cong \ft/W \\ 
(x,y,s,b_2,...,b_n) &\mapsto (b_2,...,b_n)
\end{aligned}
\end{equation}
The map $\sigma'$ is $\C^*$-equivariant if we endow the following $\C^*$-actions on $\C^3$ and $\C^{n-1}$:
\begin{equation}\label{eq:weights}
    (x,y,s)\mapsto (\lambda^2 x, \lambda^{2(n-1)}y,\lambda^2 s) , \quad (b_2,...,b_n)\mapsto (\lambda^4b_2,...,\lambda^{2n}b_n).
\end{equation}
Since the semi-universal $\C^*$-deformation of a simple singularity is unique up to isomorphism, the two deformations $\sigma$ and $\sigma'$ are isomorphic. In other words, the Slodowy slice $S$ contained in $\fg$ is isomorphic to the hypersurface $H$ in $\C^3\times \C^{n-1}$ as semi-universal $\C^*$-deformation. Note that it is important to choose the $\C^*$-action on $\C^{n-1}\cong \ft/W$ and $\C^3$ as above for $S$ and $H$ to be isomorphic as $\C^*$-deformation (see \cite[Remark 2.5.3]{beck2020folding}). 

Next, let's turn to the global situation. We again have the isomorphism of associated bundles
\begin{equation*}
    L|_{\sm^\circ}\times_{\C^*} \C^* \cong  K_{\sm}(D)|_{\sm^\circ}\times_{\C^*} \C^*
\end{equation*}
with the weights of the $\C^*$-action on the left hand side being twice the weights on the right hand side. Hence, the associated bundle $L|_{\sm^\circ}\times _{\C^*} \C^3$ is  
\begin{equation*}
    L^2|_{\sm^\circ}\oplus L^{2(n-1)}|_{\sm^\circ}\oplus L^2|_{\sm^\circ} \cong (K_\sm(D)\oplus K_\sm(D)^{\otimes n-1}\oplus K_\sm(D))|_{\sm^\circ}\cong V|_{\sm^\circ}
\end{equation*}
Also, since the elementary modification is an isomorphism i.e. $ V|_{\sm^\circ}\cong W|_{\sm^\circ}$ away from $D$, the previous construction (\ref{defining equation}) of the family $(\pi^\circ)^{ur}:(\cX^\circ)^{ur} \to B^{ur}$ as a family of hypersurfaces in the total space of $W|_{\sm^\circ}$  is equivalent to the construction as the pullback of the gluing of $H$ and $\sigma'$ over $K_\sm (D)|_{\sm^\circ}$: 
\begin{equation}\label{eq:Slodowygluing}
    \begin{tikzcd}
    (\cX^\circ)^{ur} \arrow[r] \arrow[d, "(\bm{\pi^\circ})^{ur}"] & \bm{H}|_{\sm^\circ} \subset \Tot(K_\sm(D)|_{\sm^\circ}\times_{\C^*} \C^3) \times \Tot(K_\sm(D)|_{\sm^\circ}\times_{\C^*} \ft/W )\arrow[d,"\bm{(\sigma')}|_{\sm^\circ}"] \\
    \Sigma^\circ \times B^{ur} \arrow[r, "ev"] & \bm{U}|_{\sm^\circ} =\Tot(K_\Sigma(D)|_{\Sigma^\circ} \times_{\C^*} \ft/W) 
    \end{tikzcd}
\end{equation}
where we define $\bm{\sigma'}:\bm{H} = \Tot(K_\sm(D)\times_{\C^*} H)\to \bm{U}$ and all the $\C^*$-actions in the diagram are understood as having half the weights in (\ref{eq:weights}). By the argument that $S$ and $H$ are isomorphic as $\C^*$-deformation, we have that $\bm{\sigma}|_{\sm^\circ}:\bm{S}|_{\sm^\circ}= \Tot(L|_{\sm^\circ}\times_{\C^*} S)  \to \bm{U}|_{\sm^\circ}$ and $\bm{\sigma'}|_{\sm^\circ}: \bm{H}|_{\sm^\circ}\to \bm{U}|_{\sm^\circ}$ are also isomorphic. By pulling back this isomorphism along the evaluation map to $\sm^\circ \times B^{ur}$, we get the isomorphism $(\cY^\circ)^{ur} \cong (\cX^\circ)^{ur}. $

\end{proof}

\subsection{Calabi-Yau integrable systems}
    Having constructed the family of Calabi-Yau threefolds $\pi^{ur}:\cX^{ur}\to B^{ur}$, our goal in this section is to show that the associated relative intermediate Jacobian fibration gives rise to a semi-polarized integrable system, which will again be called \textit{Calabi-Yau integrable system}. The idea is to view it as a variation of $\Z$-mixed Hodge structures of type $\{(-1,-1),(-1,0),(0,-1)\}$ and find a suitable abstract Seiberg-Witten differential as in Section 2.2. We begin by studying the $\Z$-mixed Hodge structures on both $H_3(X_b, \C)$ and $H^3(X_b, \C)$.
    \begin{proposition}\label{Hodge type argument}
        For $b \in B^{ur}$, the third homology group $H_3(X_b,\Z)$ admits a $\Z$-mixed Hodge structure of
        type  $\{(-2,-2),(-2,-1),(-1,-2)\}$. Moreover, the third cohomology group $H^3(X_b,\Z)$ admits a $\Z$-mixed Hodge structure of type $\{(1,2),(2,1),(2,2)\}$.
    \end{proposition}
    \begin{proof}
    By Poincar\'e duality, it is enough to show that the third compactly supported cohomology group $H^3_c(X_b^\circ, \Z) $ admits a $\Z$-mixed hodge structure of type $\{(1,1),(1,2),(2,1)\}$. We drop the subscript $b$ for simplicity. Recall that by Lemma \ref{CYfamilies} we have the following commutative diagram of fibrations
    \begin{equation*}
        \begin{tikzcd}
        X  \arrow[d, "\pi"] & X^\circ \arrow[l, hook' ]\arrow[r, "\cong"] \arrow[d, "{\pi}|_{X^\circ}"] & Y^\circ \arrow[r, hook] \arrow[d, "{\pi_Y}|_{Y^\circ}"] & Y \arrow[d, "\pi_Y"] \\
        \sm  & \sm^\circ \arrow[l, hook']\arrow[r, equal] & \sm^\circ \arrow[r, hook] & \sm
        \end{tikzcd}
    \end{equation*}
    
    First, consider the long exact sequence of compactly supported cohomologies associated to the pair $(X, \pi^{-1}(D))$
    \begin{equation*}
        \cdots \to H_c^2(\pi^{-1}(D),\Z) \to H_c^3(X^\circ,\Z) \to H_c^3(X,\Z) \to H_c^3(\pi^{-1}(D),\Z) \to \cdots
    \end{equation*}
    As $H_c^2(\pi^{-1}(D),\Z)=H_c^3(\pi^{-1}(D),\Z)=0$, we have an isomorphism of $\Z$-mixed Hodge structures
    \begin{equation} \label{eq:3 cohomology equal compactly supported}
        H^3_c(X,\Z) \cong H^3_c(X^\circ,\Z)
    \end{equation}
    Combining with Lemma \ref{gluingmodel}, we have the isomorphism of $\Z$-mixed Hodge structures $H^3_c(X^\circ,\Z) \cong H^3_c(Y^\circ, \Z)$.

    From now on, as we are mainly interested in the Hodge type, we work over the coefficient $\C$, instead of $\Z$. Similarly, consider the long exact sequence of compactly supported cohomologies associated to the pair $(Y, \pi_Y^{-1}(D))$
     \begin{equation*}
        \cdots \to H_c^2(Y, \C) \xrightarrow{\phi} H_c^2(\pi_Y^{-1}(D),\C) \to H_c^3(Y^\circ,\C) \to H_c^3(Y,\C) \to H_c^3(\pi_Y^{-1}(D),\C)=0 \to \cdots
    \end{equation*}
    Note that for $p \in D$,  the fiber $\pi_Y^{-1}(p)$ is deformation equivalent to a smooth fiber of the universal unfolding of $A_{n-1}$-singularity $\C^2/\Z_n$ so that $H^2_c(\pi_Y^{-1}(p),\C)$ admits a pure Hodge structure of type $(1,1)$. Also, the morphism $\phi$ is not surjective since $|D| \geq 2$. As argued in \cite[Lemma 3.1]{diaconescu2006intermediate}, $H_c^3(Y,\C)$ has a pure Hodge structure of type $\{(1,2),(2,1)\}$. Therefore, we can conclude that the $\Z$-mixed Hodge structure on $H^3_c(X,\Z)\cong H^3_c(Y^\circ, \Z)$ is of type $\{(1,1),(1,2),(2,1)\}$.
    \end{proof}
    By taking a Tate twist by $-1$, we have the $\Z$-mixed Hodge structure $H_3(X_b, \Z)(-1)$ of type $\{(-1,-1), (-1,0), (0,-1)\}$. The homology version of the second intermediate Jacobian of $X_b$ is defined to be the Jacobian associated to the $\Z$-mixed Hodge structure $H_3(X_b,\Z)(-1)$
    \begin{equation}
        J_2(X_b):=J(H_3(X_b,\Z)(-1))=\frac{H_3(X_b,\C)(-1)}{F^{-1}H_3(X_b,\C)(-1)+H_3(X_b,\Z)(-1)}
    \end{equation}
    
    \begin{rem}
    The third homology group $H_3(X_b,\Z)$  turns out to have a torsion (see Theorem \ref{localDDP}). To get the $\Z$-mixed Hodge structure on the lattice of the semi-abelian variety $J_2(X_b)$, we should consider the $\Z$-mixed Hodge structure on the torsion-free part $H_3(X_b,\Z)_{\textrm{tf}}$.
    \end{rem}
    
    \begin{corollary}
    For $b \in B^{ur}$, the homology version of the second intermediate Jacobian $J_2(X_b)$ is a semi-abelian variety. 
    \end{corollary}
    
    \begin{rem}(Adjoint Type)\label{rem:CYadjointtype}
    Unlike the holomorphic case in \cite{diaconescu2006intermediate}, the cohomological intermediate Jacobian on $H^3(X_b,\Z)$ is not a semi-abelian variety because of the Hodge type (see Appendix \ref{appendix}). This is one of the new features, so we need to consider different data to describe the case of  $PGL(n,\C)$, the adjoint group of type A. It turns out that the right object is a mixture of compactly supported cohomology and ordinary cohomology associated to $\pi_b:X_b \to \Sigma$:
    \begin{equation*}
        H^1_c(\sm,R^2\pi_{b*}\Z)\cong H^1_c(\sm^\circ,R^2\pi^\circ_{b*}\Z).
    \end{equation*}
    which can be considered a different integral structure on $H_c^3(X_b, \C)$ (see \textbf{Step 1} of the proof of Theorem \ref{localDDP}).
    \end{rem}

\begin{proposition}
The relative intermediate Jacobian fibration $\pi^{ur}: \mathcal{J}(\cX ^{ur}/B^{ur}) \to B^{ur}$ is a semi-polarized integrable system. 
\end{proposition}
\begin{proof}
By the Lemma \ref{CYfamilies}, it is enough to show that there exists an abstract Seiberg-Witten differential associated to the subfamily $(\pi^\circ)^{ur}:({\cX}^\circ)^{ur} \to B^{ur}$. In other words, we need to construct a holomorphic 3-form $\lambda_{CY}^\circ$ on $(\cX^\circ)^{ur}$ which yields the nowhere vanishing holomorphic volume form $\lambda^\circ_{CY,b} \in H^0(X_b^\circ, K_{X_b^\circ})$ for each $b \in B^{ur}$ and satisfies the condition (\ref{def:SW condition}).

First, the holomorphic 3-form $\lambda_{CY}^\circ$ is obtained from the holomorphic 3-form ${\lambda}$ on $\bm{S}$. Note that the Kostant-Kirillov form on $\fg$ induces the nowhere vanishing section in $ \nu \in H^0(S, K_{\sigma})$. One can glue the sections over $L$ by tensoring with local frames in the pullback of $K_\Sigma$, which turns out to be the holomorphic 3-form $\lambda$ on $\bm{S}$ \cite{diaconescu2006intermediate}\cite[Section 4]{beck2017hitchin}. By restricting $\lambda$ to $\Sigma^\circ$, it becomes a global holomorphic 3-form whose pullback to $(\cX^{\circ})^{ur}$ 
is the 3-form $\lambda_{CY}^\circ \in H^3((\cX^\circ)^{ur}, \C)$. By construction, for each $b \in B^{ur}$, the restriction $\lambda_{CY,b}^\circ$ becomes the nowhere vanshing holomorphic 3-form on $X_b^\circ$.
Next, the proof that ${\lambda^\circ_{CY}}$ becomes the abstract Seiberg-Witten differential relies on our main result (Theorem \ref{thm:main meromorphic DDP}). In particular, we identify the volume form ${\lambda^\circ_{CY}}$ with the abstract Seiberg-Witten differential for the Hitchin system so that the form  ${\lambda_{CY}^\circ}$ automatically satisfies the condition (\ref{def:SW condition}). Therefore, it follows from Proposition \ref{SWdifferentials and S-PIS} that $\mathcal{J}(\cX ^{ur}/B^{ur}) \to B^{ur}$ is a semi-polarized integrable system. 
\end{proof}

\section{Meromorphic DDP correspondence}
\subsection{Isomorphism of semi-polarized integrable systems }
The goal of this section is to prove an isomorphism between the two semi-polarized integrable systems that have been studied so far: the moduli space of unordered diagonally framed Higgs bundles $\cM^{\Delta}(n,D)^{ur}\to B^{ur}$ and the relative intermediate Jacobian fibration $\cJ(\cX^{ur}/B^{ur})\to B^{ur}$ of the family of Calabi-Yau threefolds $\pi^{ur}:\cX^{ur}\to B^{ur}$. The main result is stated as follows.

\begin{theorem}\label{thm:main meromorphic DDP}
There is an isomorphism of semi-polarized integrable systems:
\begin{equation}
\begin{tikzcd}
    \mathcal{J}(\mathcal{X}^{ur}/B^{ur}) \arrow[rr, "\cong"]  \arrow[rd, "\pi^{ur}"']  &&
    \cM^\Delta(n,D) ^{ur} \arrow[ld,"h_{\Delta}^{ur}"]
    \\
    & B^{ur}
\end{tikzcd}
\end{equation}
\end{theorem}

Recall that we have shown in Proposition \ref{spectralcorrespondence} and Corollary \ref{Hodge structure of Prym} that $(h_{\Delta}^{ur})^{-1}(b) \cong \Prym(\overline{\sm}_b^\circ, \sm^\circ)\cong J(H_{\Delta, SL(n), b})$ where $H_{\Delta,SL(n), b} := H_1(\Prym(\overline{\sm}^\circ_b,\sm^\circ),\Z)=H_1(\sm^\circ, \cK_b|_{\sm^\circ})_{\textrm{tf}}$ and $\cK_b:=\ker(\Tr:\overline{p}_{b*}\Z\to \Z)$. By definition, the fiber $(\pi^{ur})^{-1}(b)= J_2(X_b) = J(H_3(X_b,\Z)(-1)).$ The specialization of Theorem \ref{thm:main meromorphic DDP} to $b\in B^{ur}$ is equivalent to an isomorphism between the semi-abelian varieties $J_2(X_b)$ and $\Prym(\overline{\sm}^\circ_b, \sm^\circ)$, or equivalently, between the $\Z$-mixed Hodge structures $H_3(X, \Z)_\textrm{tf}(-1)$ and $H_{\Delta,SL(n),b}$ of type $\{(-1,-1), (-1,0), (0,-1)\}.$ We begin by proving the following result. 

\begin{theorem}\label{localDDP}
For $b\in B^{ur}$, there is an isomorphism of $\Z$-mixed Hodge structures: 
\begin{equation}
    (H_3(X_b,\Z)_{\textrm{tf}}(-1), W_\bullet^{CY}, F^\bullet_{CY}) \cong (H_{\Delta,SL(n),b}, W_\bullet^{\Delta,b}, F^\bullet_{\Delta,b}).
\end{equation}
\end{theorem}

\begin{proof}
We first fix some notations. Recall that $\sm^\circ: = \sm \setminus D$. Denote by $\sm^1:=\sm^\circ \setminus Br(\widetilde{p}_b^\circ), \widetilde{\sm}^1_b:=\widetilde{\sm}^\circ_b \setminus Ram(\widetilde{p}_b^\circ)$ the complement of the ramification and branch divisors in $\sm^\circ_b, \widetilde{\sm}^\circ_b$ respectively. Since the fibers of the cameral cover and spectral cover over each point in $\sm$ are the ordered and unordered solutions of the associated characteristic polynomial respectively \cite[Section 9]{DonagiGaitsgory}, the branch divisor of the spectral cover $\overline{p}_b^\circ :\overline{\sm}_b^\circ \to \sm^\circ$ is contained in the branch divisor of the cameral cover $\widetilde{p}_b^\circ: \widetilde{\sm}_b^\circ\to \sm^\circ$. We write $\overline{\sm}^1_b: = \overline{\sm}^\circ_b \setminus (\overline{p}^\circ_b)^{-1}Br(\widetilde{p}^\circ_b). $ The restricted maps of the spectral cover $\overline{p}^1_b:\ts^1_b\to \sm^1$ and the cameral cover $\widetilde{p}_b^1: \widetilde{\sm}_b^1\to \sm^1$ are then unramified. Similarly, we write $X_b^1\subset X_b^\circ$ the complement of $(\pi^\circ_b)^{-1}(D)$ in $X_b^\circ$ and the restricted map as $\pi^1_b: X_b^1\to \sm^1.$  

\noindent \textbf{Step 1.}
There are isomorphisms of $\Z$-mixed Hodge structures of type \\
$\{(-1,-1), (-1,0), (0,-1)\}$ 
\begin{align}
    H_3(X_b, \Z)(-1) \cong H^3_c(X_b ,\Z)(2) \cong H^3_c(X_b^\circ ,\Z)(2) \cong H^1_c(\sm^\circ, R^2\pi^\circ_{b!}\Z)(2).
\end{align}
The first two isomorphisms follow from Proposition \ref{Hodge type argument}. One can see that the last isomorphism is given by the Leray spectral sequence associated to $\pi_b^\circ:X_b^\circ \to \Sigma_b^\circ$ whose $E_2$-page is given by
        \begin{equation}
            E_2^{pq}=H_c^p(\Sigma_b^\circ, R^q\pi^\circ_{b!}\Z)
        \end{equation}
Note that a generic fiber of $\pi^\circ_b$ is deformation equivalent to a smooth fiber of the universal unfolding of $A_{n-1}$-singularity $\C^2/\Z_n$. Moreover, as we restrict to $b \in B^{ur}$, the singular fibers of $\pi^\circ_b$ have a single node and are isomorphic to a smooth fiber of $\pi^\circ_b$  with a $(-2)$-curve contracted. It implies that $R^q\pi^\circ_{b!}\Z$ is trivial for $q=1$ so that only $E_2^{12}$-term does not vanish. Therefore, the spectral sequence degenerates at $E_2$-page and we have $H^3_c(X_b^\circ ,\Z) \cong H^1_c(\sm^\circ, R^2\pi^\circ_{b!}\Z)$.
        
\noindent \textbf{Step 2.}\begin{lemma}
Over $\sm^\circ$, we have an isomorphism of sheaves,
\begin{equation}
    R^2{\pi^\circ_{b}}_! \Z \cong  (\widetilde{p}_{b*}^\circ\Lambda_{SL(n)})^W.
\end{equation}
\end{lemma}
\begin{proof} 
In the classical work of \cite{slodowy1980four}, Slodowy provided a detailed study of the topology of the maps in the following diagram via its simultaneous resolution:
\begin{equation}
    \begin{tikzcd}
     \widetilde{S}\arrow[r]\arrow[d,"\widetilde{\sigma}"]&S\arrow[d,"\sigma"]\\
     \ft\arrow[r,"\phi"]& \ft/W
    \end{tikzcd}
\end{equation}
It can be shown that there is an isomorphism of constructible sheaves 
\begin{equation}\label{isom of constructible sheaves 0}
    R^2 \sigma^1_!\Z  \cong  (\phi^1_*\Lambda _{SL(n)})^W
\end{equation}
over an open subset $\ft^1/W\subset \ft/W$ defined as the image of another open subset $\ft^1\subset \ft$ under $\phi$. Here we denote $\phi^1:= \phi|_{\ft^1}$ and $\sigma^1 : \sigma^{-1}(\ft^1/W) \to \ft^1/W$. For details, see \cite[Lemma 5.1.3]{beck2017hitchin}.

Next, we glue the maps $\sigma$ and $\phi$ along $K_\sm(D)|_{\Sigma^\circ}$ as in  (\ref{eq:Cameralgluing}) and (\ref{eq:Slodowygluing}) 
\begin{equation}\label{gluing of simulateneous resolution}
    \begin{tikzcd}
    &\bm{S}|_{\sm^\circ}  \arrow[d,"\bm{\sigma}|_{\sm^\circ}"] \\
    \bm{\widetilde{U}}|_{\sm^\circ}\arrow[r,"\bm{\phi}|_{\sm^\circ}"]&\bm{U}|_{\sm^\circ}
    \end{tikzcd}
\end{equation}
where we define $\bm{\widetilde{U}}|_{\sm^\circ}:= \Tot(K_\sm(D)|_{\sm^\circ}\times_{\C^*}\ft)$ and we recall that $\bm{S}|_{\sm^\circ} = \Tot(L|_{\sm^\circ} \times _{\C^*} S)$ and $\bm{U}|_{\sm^\circ}= \Tot(K_\sm(D)|_{\sm^\circ} \times_ {\C^*} \ft/W) \cong \Tot(L|_{\sm^\circ}\times_{\C^*} \ft/W)$. 
Let us define $\bm{U^1}:=  \Tot(K_\sm(D)\times_{\C^*}\ft^1/W)\subset \bm{U}$. 
Since the varieties here are glued using the same cocyle of $L|_{\sm^\circ}$ (again, in taking the associated bundles here, $L|_{\sm^\circ}$ as a $\C^*$-bundle acts with twice the weights of the action by $K_{\sm}(D)|_{\sm^\circ}$), the isomorphism of constructible sheaves (\ref{isom of constructible sheaves 0}) over $\ft^1/W$ also glues together to another isomorphism of constructible sheaves over $\bm{U^1}|_{\sm^\circ}$:
\begin{equation}\label{isom of constructible sheaves 1}
    R^2 (\bm{\sigma})_!\Z \cong (\bm{\phi}_{*}\Lambda_{SL(n)})^W.
\end{equation}
As argued in Lemma \ref{gluingmodel}, $\bm{\sigma}|_{\sm^\circ}:\bm{S}|_{\sm^\circ} \to \bm{U}|_{\sm^\circ}$ is equivalent to $\bm{\sigma'}|_{\sm^\circ}:\bm{H}|_{\sm^\circ} \to \bm{U}|_{\sm^\circ}$, so we obtain 
\begin{equation}\label{isom of constructible sheaves 2}
    R^2 (\bm{\sigma'})_!\Z \cong (\bm{\phi}_{*}\Lambda_{SL(n)})^W
\end{equation}
over $\bm{U^1}|_{\sm^\circ}$. In both (\ref{isom of constructible sheaves 1}) and (\ref{isom of constructible sheaves 2}), we drop the notation of the restrictions of $\bm{\sigma}$, $\bm{\sigma'}$ and $\bm{\phi}$ to $\bm{U^1}|_{\sm^\circ}$ for convenience. 

Recall from Lemma \ref{gluingmodel} that $\pi_b^\circ: X_b^\circ \to \sm^\circ$ can be obtained by pulling back from $\bm{\sigma'}|_{\sm^\circ}:\bm{H}|_{\sm^\circ}\to \bm{U}|_{\sm^\circ}$ along the composition of the inclusion and the evaluation map $\sm^\circ\times \{b\} \hookrightarrow \sm^\circ\times B\to \bm{U}|_{\sm^\circ}$. For $b\in B^{ur}$, the section $b:\sm \to \bm{U}$ factorizes through $\bm{U^1}$ and then restricts to $b|_{\sm^\circ}: \sm^\circ \to \bm{U^1}|_{\sm^\circ}$, so the isomorphism (\ref{isom of constructible sheaves 2}) specializes to $R^2\pi^\circ_{b!}\Z\cong (\widetilde{p}_{b*}^\circ \Lambda_{SL(n)})^W$ by pulling back along $b|_{\sm^\circ}.$ 

\end{proof}

\noindent\textbf{Step 3.}\begin{lemma}\label{lem:spectraltocameral} Over $\sm^\circ$, we have an isomorphism of sheaves, 
\begin{align}
(\widetilde{p}^\circ_{b*}\Lambda_{SL(n)})^W \cong  \cK_b|_{\sm^\circ}.    
\end{align}
\end{lemma}

\begin{proof}

To simplify the notation, we will write $\cK^\circ_b:= \cK_b|_{\sm^\circ}$ in this proof.
Recall that there is an isomorphism (see \cite[(6.5)]{donagi1993decomposition}) between the two sheaves away from the branch locus:
\begin{equation}\label{spectralvscameral}
    \overline{p}_{b*}^1\Z \cong (\widetilde{p}_{b*}^1 R)^W 
\end{equation}
where $R:=\Z[W/W_0]$ denote the free abelian group generated by the set of right (or left) cosets $W/W_0$ and $W_0$ is a subgroup of $W=S_n$ isomorphic to $S_{n-1}$ that fixes a chosen element of $\{1,..,n\}$ (on which $W=S_n$ acts).\footnote{The choice of the element in $\{1,...,n\}$ or equivalently the choice of subgroup $W_0\cong S_{n-1}$ in $W$ is not important here as $\widetilde{\sm} /W_0 \cong \ts$ for any choice of $W_0$ (\cite[Section 9]{DonagiGaitsgory}). } 
Then we see that 
\begin{equation*}
    \cK_b|_{\sm^1}  = \ker( \overline{p}^1_{b*}\Z \to \Z) \cong  \ker ((\widetilde{p}_{b*}^1R)^W \to \Z) \cong (\widetilde{p}^1_{b*}\Lambda_{SL(n)} )^W ,
\end{equation*}
the last isomorphism holds because $\ker(R\to \Z) = \Lambda_{SL(n)}$. 

Denote by $j: \sm^1\to \sm^\circ $ the inclusion map. We first write $\cK_b^\circ$ as $j_*j^*\cK_b^\circ $. Indeed, as $\overline{p}^\circ_{b*}\Z = j_*\overline{p}^1_{b*}\Z $ and $\Z\cong j_*\Z$, applying the functor $j_*$ to the short exact sequence $0\to j^*\cK_b^\circ\to \overline{p}^1_{b*}\Z\xrightarrow{\Tr|_{\sm^1}} \Z\to 0$, we get
\begin{equation*}
    0\to j_*j^*\cK_b^\circ \to j_*\overline{p}_{b*}^1\Z = \overline{p}_{b*}^1\Z \xrightarrow{\Tr} j_*\Z= \Z \to R^1j_*j^*\cK_b^\circ \to... 
\end{equation*}
In particular, it follows that $j_*j^*\cK_b^\circ \cong \ker(\Tr) = \cK_b^\circ.$

Hence, we get
\begin{equation*}
    (\widetilde{p}_{b*}^\circ \Lambda_{SL(n)})^W \cong j_*(\widetilde{p}_{b*}^1\Lambda_{SL(n)})^W  \cong j_* j^*\cK_b^\circ \cong \cK_b^\circ
\end{equation*}
which means that the isomorphism (\ref{spectralvscameral}) above extends from $\sm^1$ to $\sm^\circ.$
\end{proof}

\noindent\textbf{Step 4.} Finally, we have the isomorphisms of sheaves 
\begin{equation*}
    R^2\pi_{b!}^\circ \Z\cong (\widetilde{p}_{b*}^\circ \Lambda_{SL(n)})^W\cong \cK_b|_{\sm^\circ}
\end{equation*} 
which induce the isomorphisms of $\Z$-mixed Hodge structure of type $\{(0,0),(0,1),(1,0)\}$ on 
\begin{equation*}
    H^1_c(\sm^{\circ}, R^2\pi_{b!}^\circ \Z)(1) \cong H^1_c( \sm^\circ, (\widetilde{p}_{b*}^\circ \Lambda_{SL(n)})^W) \cong H^1_c(\sm^\circ, \cK_b|_{\sm^\circ}).
\end{equation*}

Hence, by taking the torsion free part and Tate twists, we achieve the isomorphism of $\Z$-mixed Hodge structures of type $\{(-1,-1,), (-1,0), (0,-1)\}$
\begin{center}
    $H_3(X_b,\Z)_\textrm{tf} (-1) \cong H^1_c(\sm^\circ, \cK_b|_{\sm^\circ})_\textrm{tf}(2) \cong H_{\Delta,SL(n),b}.$ 
\end{center}
\end{proof}
 By the equivalence between semi-abelian varieties and torsion free $\Z$-mixed Hodge structures of type $\{(-1,-1), (-1,0)$, $(0,-1)\}$, we immediately get the following result: 
\begin{corollary}
We have an isomorphism of semi-abelian varieties
\begin{equation}
    J_2(X_b) \cong  h^{-1}_{\Delta}(b) \cong \Prym(\ts^\circ_b, \sm^\circ).
\end{equation}
\end{corollary}

Now we return to the main theorem.  
\begin{proof}[Proof of Theorem \ref{thm:main meromorphic DDP}]
Clearly, the argument in Theorem \ref{localDDP} works globally for the family of Calabi-Yau threefolds $\pi^{ur}=pr_2\circ\bm{\pi}^{ur}:\cX^{ur} \to \sm\times B^{ur}\to   B^{ur}$ and the family of punctured spectral curves $pr_2\circ\bm{\overline{p}}^{ur}: {\bm{\overline{\Sigma}^\circ}}\to \sm^\circ\times B^{ur}\to B^{ur}$, so it yields an isomorphism of variations of $\Z$-mixed Hodge structures: 
\begin{equation}
    R^3(pr_2\circ\bm{\pi}^{ur})_! \Z(-1)  \cong R^1(pr_2)_!( \bm{\cK})
\end{equation}
where $\bm{\cK}:= \ker(\bm{\Tr}: \bm{\overline{p}}^{ur}_*\Z \to \Z). $
By taking the relative Jacobian fibrations of both sides, we immediately get an isomorphism of varieties: 
\begin{equation}
        \begin{tikzcd}
     \mathcal{J}(\mathcal{X}^{ur}/B^{ur}) \arrow[rr, "\cong"]  \arrow[rd, "\pi^{ur}"']  && \bm{\Prym}  (\bm{\overline{\sm}^\circ},\sm^\circ) \cong \cM^{\Delta}(n,D)^{ur} \arrow[ld, "h_{\Delta}^{ur}"] \\
    & B^{ur}
\end{tikzcd}
\end{equation}
where $\bm{\Prym}  (\bm{\overline{\sm}^\circ},\sm^\circ)$ is the relative Prym fibration of the family of punctured spectral curves $\bm{\overline{\sm}^\circ}\to B^{ur}$. 
By the spectral correspondence proved in Proposition \ref{spectralcorrespondence}, we have $\bm{\Prym} (\bm{\overline{\sm}^\circ},\sm^\circ)\cong \cM^{\Delta}(n,D)^{ur}$. 

It remains to verify that the morphism $\mathcal{J}(\mathcal{X}^{ur}/B^{ur})\to \cM^{\Delta}(n,D)^{ur}$ intertwines the abstract Seiberg-Witten differentials constructed on each side. 
This can be easily obtained by modifying the results in \cite{diaconescu2006intermediate} \cite{beck2017hitchin} to our punctured case. Note that both the abstract Seiberg-Witten differentials come from the tautological section on $\bm{\widetilde{U}}$. In order to compare them, 
we again look at the simulteneous resolution of $S\to \ft/W$:
\begin{equation}
    \begin{tikzcd}
    \widetilde{S} \arrow[r] \arrow[d, "\widetilde{\sigma}"] & S \arrow[d, "\sigma"] \\
    \ft \arrow[r, "\phi"] & \ft/W
    \end{tikzcd}
\end{equation}
and recall that $\widetilde{\sigma}$ is $C^\infty$-trivial.

Taking a step further in (\ref{gluing of simulateneous resolution}), we can glue all the maps in the simultaneous resolution diagram to a commutative diagram
\begin{equation}
    \begin{tikzcd}
         \bm{\widetilde{S}}|_{\sm^\circ}\arrow[r,"\bm{\Psi}"]\arrow[d,"\bm{\widetilde{\sigma}}|_{\sm^\circ}"] &\bm{S}|_{\sm^\circ}\arrow[d,"\bm{\sigma}|_{\sm^\circ}"]\\
         \bm{\widetilde{U}}|_{\sm^\circ}\arrow[r,"\bm{\phi}|_{\sm^\circ}"] &\bm{U}|_{\sm^\circ}
    \end{tikzcd}
\end{equation}
where $\bm{\widetilde{S}}|_{\sm^\circ}:=\Tot(L|_{\sm^\circ}\times _{\C^*} \widetilde{S})$. It induces the following diagram by pulling back under the evaluation map, 
\begin{equation*}
    \begin{tikzcd}
         (\widetilde{\cX^\circ})^{ur}  \arrow[d, "(\widetilde{\pi^{\circ}})^{ur}"] \arrow[r, "\Psi"] & (\cX^\circ)^{ur} \arrow[d, "(\pi^\circ)^{ur}"] \\
         B^{ur} \arrow[r, equal] & B^{ur}
    \end{tikzcd} 
\end{equation*}
Then, $\Psi$ induces an inclusion
\begin{equation}
    \Psi^*: R^3(\pi^\circ)^{ur}_*\C \otimes \cO_{B^{ur}} \to  R^3(\widetilde{\pi^\circ})^{ur}_*\C \otimes \cO_{B^{ur}}
\end{equation}
so that we can lift $\lambda_{CY}^{\circ}$ to $(\widetilde{\cX^\circ})^{ur}$. As both are induced from the tautological section on $\bm{\widetilde{U}}$, under the following isomorphism
 \begin{equation*}
     R^3(\widetilde{\pi^\circ})^{ur}_*\C \otimes \cO_{B^{ur}} \cong R^1(f^\circ)^{ur}_*\ft \otimes \cO_{B^{ur}}
 \end{equation*}
where $(f^\circ)^{ur} : (\bm{\cs^\circ})^{ur}\to B^{ur} $ is the family of puntured cameral curves, the two abstract Seiberg-Witten differentials $\lambda_{CY}^{\circ}$ and $ \lambda_\Delta$ coincide \cite[Theorem 5.2.1]{beck2017hitchin}. 

\end{proof}

\begin{rem}(Adjoint type)
    The above argument is easily applied to the adjoint case, $PGL(n,\C)$, so that there is an isomorphism between (unordered) diagonally framed $PGL(n,\C)$-Hitchin system and Calabi-Yau integrable system. On the Hitchin side, we consider the dual Prym sheaf $\cK^\vee$. The key is to construct the relevant family of semi-abelian varieties on the Calabi-Yau side as mentioned in Remark \ref{rem:CYadjointtype}. 
\end{rem}

\appendix
\section{Summary of Deligne's theory of
  1-motives}\label{appendix}

In \cite{TheoriedeHodge}, Deligne gave a motivic description of variations of (polarized) $\Z$-mixed Hodge structures of type $\{(-1,-1), (-1,0), (0,-1), (0,0)\}$.
We recall the arguments in \cite{TheoriedeHodge} and study the special case which is of main interest in this paper. 

\begin{definition}
 An 1-motive $M$ over $\C$ consists of
 \begin{enumerate}
     \item $X$ free abelian group of finite rank, a complex abelian variety $A$, and a complex affine torus $\mathsf{T}$. 
     \item A complex semi-abelian variety $G$ which is an extension of $A$ by $\mathsf{T}$.
     \item A homomorphism $u:X \to G$.
 \end{enumerate}
\end{definition}
We will denote a 1-motive by $(X,A,T,G,u)$ or $M=[X \xrightarrow{u}G]$. 

\begin{proposition}\label{prop:Deligne}
The category of (polarizable) mixed Hodge structures of type \\
$\{(-1,-1), (-1,0), (0,-1), (0,0)\}$ is equivalent to the category of 1-motives. 
\end{proposition}
\begin{proof}
    Given a 1-motive $M$, Deligne constructed a mixed Hodge structure $(T(M)_\Z, W,F)$ of type $\{(-1,-1), (-1,0), (0,-1), (0,0)\}$ as follows. Define a lattice $T(M)_\Z$ as the fiber product 
\begin{equation}
    \begin{tikzcd}
    T(M)_\Z \arrow[r, "\beta"] \arrow[d, "\alpha"] & X \arrow[d, "u"]  \\
    Lie(G) \arrow[r, "\exp"] & G 
    \end{tikzcd}
\end{equation}
 
The weight filtration on $T(M)_\Z$ is given by setting
$W_{-1}T(M)_\Z:=H_1(G,\Z)=\ker(\beta)$ and $W_{-2}T(M)_\Z=H_1(T,\Z)$. Also, by linearly extending $\alpha:T(M)_\Z \to Lie(G)$ to $\C$, we define $F^0(T(M)_\Z \otimes \C):=\ker(\alpha_\C)$. By construction $\Gr_{-1}^W(T(M)_\Z)=H_1(A,\Z)$ with the usual Hodge filtration and is therefore polarizable.

Conversely, if $H:=(H_\Z, W, F)$ is a mixed Hodge structure of the given type with $\Gr_{-1}^W(H_\Z)$ polarizable, then one can construct a 1-motive by taking
\begin{enumerate}
    \item $A:= \Gr_{-1}^W(H_\C) / (F^0\Gr_{-1}^W(H_\C)+\Gr_{-1}^W(H_\Z))$
    \item $\mathsf{T}:=\Gr_{-2}^W(H_\C)/\Gr_{-2}^W(H_\Z)$
    \item $G:=H_\C/(F^0H_\C+H_\Z)$
    \item $X:=\Gr_0^W(H_\Z)$
\end{enumerate}
\end{proof}

In particular, if $X$ is trivial, the 1-motive $M$ is equivalent to a semi-abelian variety $G$. By Proposition \ref{prop:Deligne}, we have an equivalence between the abelian category of semi-abelian varieties and the abelian category of (polarizable) $\Z$-mixed Hodge structures of type $\{(-1,-1), (-1,0), (0,-1)\}$. 

\begin{example}
A typical example coming from geometry is the mixed Hodge structure on the first homology group of a punctured curve. Let $C$ be a Riemann surface and $D \subset C$ be a reduced divisor. The first homology group $H_\Z=H_1(C \setminus D, \Z)$ carries a $\Z$-mixed Hodge structure of type $\{(-1,-1),(-1,0),(0,-1)\}$ where $\Gr^W_{-1}(H_\C)=H_1(C,\Z)\otimes \C$. Moreover, it admits a degenerate intersection pairing $Q:H_\Z \times H_\Z \to \Z$ whose kernel is $W_{-2}H_\C \cap H_\Z$. Note that it induces a polarization on $\Gr^W_{-1}(H_\C)$ and so gives rise to the type of object in proposition \ref{prop:Deligne}. In other words, we get a semi-abelian variety $G$ by taking the Jacobian of $(H_\Z, W_\bullet, F^\bullet)$ as follows
\begin{equation*}
\begin{aligned}
    & G:=J(H)=H_\C/(F^{0}H_\C + H_\Z) \\
    & A:=J_{\text{cpt}}(H)=\Gr^W_{-1}H_\C/(\Gr^W_{-1}F^{0}H_\C +H_\Z) \\
    & \mathsf{T}:=W_{-2}H_\C/W_{-2}H_\Z\\
\end{aligned}
\end{equation*}

\end{example}
We call such integral mixed Hodge structure a \textit{semi-polarized $\Z$-mixed Hodge structure}. 
Moreover, consider the dual mixed Hodge structure $H^\vee$ which is of type $\{(0,1), (1,0), (1,1)\}$. Geometrically it corresponds to the first cohomology $H^1(C \setminus D)$ of the punctured Riemann surface $C \setminus D$. The associated Jacobian $J(H^\vee)=H^\vee_\C/(F^1H^\vee_\C+H_\Z)$ is no longer a semi-abelian variety, but just a complex torus.

\section*{Acknowledgement}
We would like to thank our advisors Ron Donagi and Tony Pantev for their encouragement, constant support on this project and all the useful comments on the earlier drafts of this paper. We are also grateful to Rodrigo Barbosa and Chenglong Yu for many useful discussions. We would also like to thank the referees for their useful comments and pointed out that the reference for a similar construction of Calabi-Yau threefolds used in our paper of which we were not previously aware.


\printbibliography

\end{document}